\documentclass[final,leqno]{siamltex704}

\usepackage{amsmath,amssymb}
\usepackage{amsfonts}
\usepackage{exscale}
\usepackage{graphicx}
\usepackage{comment}
\usepackage{subcaption}
\usepackage[ruled]{algorithm2e}
\usepackage{algorithmic}
\usepackage{url}
\usepackage{appendix}
\usepackage{epstopdf} 
\usepackage[pdftex]{thumbpdf}      
\usepackage[pdftex,                
pdfstartview=FitH,%
bookmarks=true,
bookmarksnumbered=true,
bookmarksopen=true,%
hypertexnames=false,
breaklinks=true,
  colorlinks=true,%
  linkcolor=blue,anchorcolor=blue,%
  citecolor=blue,filecolor=blue,%
  menucolor=blue]%
{hyperref} 
\usepackage{etoolbox}

\makeatletter
\def\@cite#1#2{{\rm [}{{\rm#1}\if@tempswa , #2\fi}{\rm ]}}
\makeatother

\newcommand{\eq}[1]{\begin{equation}\label{#1}}
\newcommand{\en}{\end{equation}}

\def\IC{\mathbb{C}}
\def\inv{^{-1}}

\title{Generalized preconditioned locally harmonic residual method for non-Hermitian eigenproblems
\thanks{%
Version generated on \today. This material is based upon work supported by the 
U.S. Department of Energy, Office of Science, 
Advanced Scientific Computing Research and Basic Energy Sciences, SciDAC program
(the first and second authors) 
and National Science Foundation under Grant DMS-1419100 (the third author). 
}
}

\author{Eugene Vecharynski$^{\dagger}$ \and 
Chao Yang\thanks{Computational Research Division, Lawrence Berkeley National Laboratory,
One Cyclotron Road, MS 50F-1620L, Berkeley, CA 94720, USA ({\tt \{evecharynski,cyang\}@lbl.gov})} 
\and 
Fei Xue\thanks{Department of Mathematics, University of Louisiana at Lafayette, P.O. Box 41010, Lafayette, LA 70504-1010, USA ({\tt fxue@louisiana.edu})}}

\begin{document}
\maketitle



\setcounter{page}{1}

\begin{abstract}
We introduce the Generalized Preconditioned Locally Harmonic Residual (GPLHR) method for solving standard and generalized non-Hermitian eigenproblems. The 
method is particularly useful for computing  a subset of eigenvalues, and their eigen- or Schur vectors, closest to a given shift. The proposed method is based on block iterations and can take advantage of a preconditioner if it is 
available. It does not need to perform exact shift-and-invert transformation. 
Standard and generalized eigenproblems are handled in a unified framework.
Our numerical experiments demonstrate that GPLHR is generally more robust and efficient than existing methods, especially if the available memory is limited.  
\end{abstract}



%

\pagestyle{myheadings}
\thispagestyle{plain}
\markboth{EUGENE VECHARYNSKI, CHAO YANG AND FEI XUE}{THE GPLHR METHOD FOR NON-HERMITIAN EIGENPROBLEMS} 

\section{Introduction}\label{sec:intro}
Large non-Hermitian eigenproblems arise in a variety of important applications, 
including resonant state calculation~\cite{Balslev.Combes:71, Ho:83, Reinhardt:82} or excited state analysis for    
equation-of-motion coupled-cluster (EOM-CC) approaches~\cite{HeadGordon.Lee:97, el-str-book-Helgaker.Jorgensen.Olsen:2000} in 
quantum chemistry, linear stability analysis of the Navier-Stokes equation in fluid dynamics \cite{Cliffe.Spence.Tavener:08, eswbook:14}, 
crystal growth simulation~\cite{Langer:80, Langer.Muller:78}, magnetohydrodynamics~\cite{Kerner:86},
power systems design~\cite{Kamwa.Lefebvre.Loud:02}, and many others; see, e.g., \cite{Trefethen.Embree.05} for more examples. 

In their most general form, these problems can be written as
\eq{eq:gevp}
A x = \lambda B x,
\en
where $A$ and $B$ are general square matrices.
We are particularly interested in the case in which 
both $A$ and $B$ are very large
and sparse, or available only implicitly through a matrix-vector 
multiplication procedure.
If $B$ is the identity matrix, then~\eqref{eq:gevp} becomes a standard eigenproblem. 

The spectrum of~\eqref{eq:gevp}, denoted by $\Lambda(A,B)$, is given by a set of numbers $\lambda \in \mathbb{C}$ that make $A-\lambda B$ singular. 
The value of $\lambda$ can be infinity in the case of a singular $B$. 
Given a scalar $\sigma \in \IC$, which we refer to as a \textit{shift}, 
we seek to find a subset of eigenvalues $\lambda \in \Lambda(A,B)$ that 
are closest to $\sigma$, and their associated (right) eigenvectors~$x$.
These eigenvalues can be either extreme eigenvalues (e.g., eigenvalues
with the largest magnitude) or interior eigenvalues that are inside
the convex hull of the spectrum. 



Our focus in this paper is on algorithms for computing interior eigenvalues
and their corresponding eigenvectors of a non-Hermitian pencil.
It is well known that these eigenpairs are often difficult to compute
in practice.  Traditional methods, such as the inverse subspace iteration 
or variants of the shift-invert Arnoldi algorithm, see, 
e.g.,~\cite{templates:00, Saad-book3}, 
rely on using spectral transformations that require performing
LU factorizations of $A - \sigma B$ and computing solutions to triangular
linear systems.  Such an approach can be prohibitively expensive,
especially when the problem size is large and the LU factors of 
$A - \sigma B$ are not sparse. 

There has been some effort in recent years to develop methods that 
are factorization free.  Examples of such methods 
include the inexact inverse subspace iteration~\cite{Robbe.Sadkane.Spence:09, Xue.Elman:11} and inexact shift-invert Arnoldi methods \cite{Freitag.Spence.09, Xue.Elman:12} in which linear systems of the form $(A - \sigma B)w = r$ are
solved iteratively.  While these schemes can be considerably less expensive 
per iteration, the overall convergence of these methods is less 
predictable. There is often a tradeoff between the accuracy of 
the solution to linear~systems and the convergence rate of the 
subspace or Arnoldi iterations.  Setting~an appropriate convergence
criterion for an iterative solution of $(A - \sigma B)w = r$ is 
not straightforward. 

Another class of factorization-free methods include
the Generalized Davidson (GD) method~\cite{Morgan:92, Morgan.Scott:86} 
and the Jacobi--Davidson (JD) methods~\cite{Fokkema:Sleijpen.Vorst:98}. 
The GD methods generally do not rely on solution of linear systems. 
The JD style QR and QZ algorithms (JDQR and JDQZ) presented 
in~\cite{Fokkema:Sleijpen.Vorst:98} can be viewed as 
\textit{preconditioned eigensolvers} in which a Newton-like 
correction equation is solved approximately by a preconditioned
iterative solver.  

The GD method can be easily extended into a block method in which
several eigenpairs can be approximated simultaneously~\cite{Sadkane:93, Zuev.Ve.Yang.Orms.Krylov:15}.  Block GD is widely used in quantum chemistry.
A block GD method is particularly well suited 
for modern high performance computers with a large number of processing 
units. This is because in a block method several (sparse) matrix vector 
multiplications, which often constitute the major cost of the algorithm, 
can be carried out
simultaneously.  Furthermore, one can easily implement data blocking, 
exploit data reuse and take advantage of BLAS3 in a block method.  These techniques can lead to significant speedups on modern high performance 
computers as reported in~\cite{Aktulga.Buluc.Williams.Yang:14}. 
However, the convergence of the block GD method depends to a large 
degree on the effectiveness of the preconditioner. In quantum 
chemistry applications, the matrix $A$ is often diagonal dominant.
The diagonal part of the matrix can be used to construct 
an effective preconditioner. In other applications, it is less clear
how to construct a good preconditioner when $A$ is not diagonal 
dominant~\cite{Ve.Kn:14TR}.

The JDQR and JDQZ methods tend to be more robust with respect to 
the choice of preconditioners. However, extending JDQR and JDQZ to 
block methods is not straightforward.
Because JDQR and JDQZ methods compute one eigenpair at a time, 
concurrency can only be exploited within a single matrix-vector multiplication
procedure, which limits its parallel scalability to a large number of 
processors.


In this paper, we present a new family of methods for solving non-Hermitian 
eigenproblems, called the Generalized 
Preconditioned Locally Harmonic Residual (GPLHR) methods. The proposed 
scheme is a block method that constructs approximations to a partial 
(generalized) Schur form of the matrix pair $(A,B)$ associated 
with~\eqref{eq:gevp} iteratively.
It does not require performing an LU factorization of $A - \sigma B$. 
It allows a preconditioner to be used to construct a search 
space from which approximations to the Schur vectors can be extracted. 
We demonstrate that the GPLHR methods are generally 
more robust and efficient than JDQR/JDQZ as well as the block GD method 
when the same preconditioner is used, and when dimension of the
search spaces are comparable. 
In addition, because GPLHR is a block method, it is
well suited for high performance computers that consist of many cores or processors.
              
The GPLHR algorithm is a generalization of the recently introduced 
Preconditioned Locally Harmonic Residual (PLHR) method for computing interior eigenpairs of Hermitian eigenproblems~\cite{Ve.Kn:14TR} to the non-Hermitian case 
(hence the name). The GPLHR scheme can also be viewed as an extension of 
the well known LOBPCG~\cite{Knyazev:01} method and block inverse-free 
preconditioned (BIFP) Krylov 
subspace methods~\cite{Golub.Ye:02, Quillen.Ye:10} for computing extreme eigenpairs of Hermitian problems. 
At each step, all these methods construct 
trial subspaces of a fixed dimension and use them to extract the desired approximations by means of a subspace projection. 
The key difference 
is that GPLHR systematically uses a \textit{harmonic Rayleigh--Ritz} procedure~\cite{Morgan:91, Morgan.Zeng:98}, which allows for a robust eigen- or Schur pair extraction independent of the location
of the targeted eigenvalues.   Moreover, the proposed method explicitly utilizes properly chosen projectors to stabilize convergence, and is capable of performing iterations using the approximations to Schur vectors instead of potentially ill-conditioned
eigenvectors.

This paper is organized as follows. In Section~\ref{sec:eig}, we review standard eigenvalue-revealing decompositions for non-Hermitian
eigenproblems and propose a new generalized Schur form which emphasizes only the right Schur vectors. This form is then used to derive 
the GPLHR algorithms for standard and generalized eigenproblems in Section~\ref{sec:pqr}.
The preconditioning options are discussed in Section~\ref{sec:prec}.      
Section~\ref{sec:num} presents numerical results. Our conclusions can be found in Section~\ref{sec:con}. Appendix A briefly describes the approximate eigenbasis based GPLHR 
iteration (GPLHR-EIG).

\section{Eigenvalue-revealing decompositions}\label{sec:eig}
For non-Hermitian eigenproblems, it is common (see, e.g.,~\cite{Moler.Stewart:73, Saad-book3, Stewart.Sun:90}) to 
replace~\eqref{eq:gevp} by an equivalent~problem
\eq{eq:gevp_mod}
\beta A x = \alpha B x,
\en
where the pair $(\alpha, \beta)$, called a \textit{generalized eigenvalue}, defines an eigenvalue $\lambda = \alpha / \beta$ of the matrix pair $(A,B)$. 
This formulation is advantageous from the numerical perspective. For example, 
it allows revealing 
the infinite or indeterminate eigenvalues of~\eqref{eq:gevp}. 
The former corresponds to the case where $B$ is singular and 
$\beta = 0$, $\alpha \neq 0$, 
whereas the latter is indicated by $\alpha = \beta = 0$ 
and takes place if both $A$ and $B$ are singular with a non-trivial intersection of their null spaces. 
Additionally, 
the separate treatment of 
$\alpha$ and $\beta$ allows one to handle situations where either of the quantities is close to zero,
which leads to underflow or overflow for $\lambda = \alpha / \beta$.
Note that a generalized eigenvalue $(\alpha, \beta)$ is invariant with respect to multiplication by a scalar, i.e., for any nonzero 
$t \in \mathbb{C}$,  $(t\alpha, t\beta)$ is also a generalized eigenvalue. In order to fix a representative pair, a normalizing condition should
be imposed, e.g.,
$|\alpha|^2 + |\beta|^2 = 1$.
 
  
In this paper, we assume that the pair $(A,B)$ is \textit{regular}, i.e., $\mbox{det}(\beta A + \alpha B)$ is not identically zero for all $\alpha$ and 
$\beta$. 
The violation of this assumption gives a \textit{singular} pair, which corresponds to an ill-posed problem.  
In this case, a regularization
procedure should be applied to obtain a nearby regular pair, e.g.,~\cite[Chapter 8.7]{templates:00}, for which the eigenproblem is solved.
%
In most cases, however, $(A,B)$ is regular, even when $A$ or $B$ or both are singular as long as the intersection of their null spaces is zero.
Thus, regular pairs can have infinite eigenvalues, whereas the possibility of indeterminate eigenvalues is excluded.

\subsection{Partial eigenvalue decomposition and generalized Schur form}\label{subsec:eig}
Let us assume that the targeted eigenvalues of~\eqref{eq:gevp} are non-deficient, i.e., their algebraic and geometric multiplicities coincide. 
Then, given formulation~\eqref{eq:gevp_mod}, the \textit{partial eigenvalue decomposition} of $(A,B)$ can be written in the form
\eq{eq:partial}
A X \Lambda_B = BX \Lambda_A,
\en
where $X \in \IC^{n \times k}$ is a matrix whose columns $x_j$ correspond to the $k$ eigenvectors associated with the 
wanted eigenvalues $\lambda_j$. The $k$-by-$k$ matrices $\Lambda_A$ and $\Lambda_B$ are diagonal, with the diagonal entries 
given by $\alpha_j$ and $\beta_j$, respectively. In our context, it is assumed that $\lambda_j = \alpha_j / \beta_j$ are the eigenvalues 
closest to $\sigma$.

Since, by our assumption, $\lambda_j$'s are non-deficient, there exists $k$  associated linearly independent eigenvectors, so that $X$ is full-rank and~\eqref{eq:partial}
is well defined. 
For a vast majority of applications, the partial eigenvalue decomposition~\eqref{eq:partial} does exist, i.e., the desired eigenvalues are
non-deficient. Nevertheless, pursuing decomposition~\eqref{eq:partial} is generally not recommended, as it aims at computing a
basis of eigenvectors which can potentially be ill-conditioned. 

In order to circumvent the difficulties related to ill-conditioning or deficiency of the eigenbasis,
we consider instead a \textit{partial generalized Schur decomposition}
\eq{eq:gschur}
\left\{
\begin{array}{ccc}
A V & = & Q R_A, \\
B V & = & Q R_B; 
\end{array} 
\right.
\en
see, e.g.,~\cite{Fokkema:Sleijpen.Vorst:98, Saad-book3, Stewart:71}.
%
Here, $V$ and $Q$ are the $n$-by-$k$ matrices of the \textit{right and left generalized Schur vectors}, respectively. 
The factors $R_A$ and $R_B$ are $k$-by-$k$ upper triangular
matrices, such that the ratios $\lambda_j = R_A(j,j)/R_B(j,j)$ of their diagonal entries correspond to the desired eigenvalues.

The advantage of the partial Schur form~\eqref{eq:gschur} is that it is defined for any 
$(A,B)$ and both $V$ and $Q$ are orthonormal,
so that they 
can be computed in a numerically stable manner.
The matrix $V$ then gives an orthonormal basis of the 
invariant subspace associated with the eigenvalues of interest. 
If individual eigenvectors are needed, the right generalized Schur vectors can be transformed into 
eigenvector block $X$ 
by the means of the standard Rayleigh--Ritz procedure for the pair $(A,B)$, performed over the subspace spanned by the columns of $V$.  
Similarly,~$Q$ contains an orthonormal basis of the left invariant subspace and can be used to retrieve the left eigenvectors. 


\subsection{The ``$Q$-free'' partial generalized Schur form}\label{subsec:gschur_new}

We start by addressing the question of whether it is possible to eliminate $Q$ from
decomposition~\eqref{eq:gschur} and obtain 
an eigenvalue-revealing Schur type form that contains only the right Schur vectors $V$
and triangular factors. 
%
Such a form would provide a numerically stable alternative to the eigenvalue decomposition~\eqref{eq:partial}. 
It would also give an opportunity to handle the standard and generalized 
eigenproblems in a uniform manner.
Most importantly, 
it would allow us to define a numerical scheme that 
emphasizes computation only of
the generalized right Schur vectors, which are often the ones needed in a real application.
The left Schur 
basis $Q$ can be obtained as a by-product of~computation.

For example, if $R_B$ is nonsingular, then $Q$ can be eliminated in a straightforward manner by multiplying the bottom equation in~\eqref{eq:gschur} by $R_B^{-1}$ and substituting
the resulting expression for $Q$ into the top equation. This yields the desirable decomposition $AV = BV M_A$, where $M_A = R_B^{-1} R_A$ is triangular with eigenvalues on the diagonal.
Similarly, assuming that $R_A$ is nonsingular, one can obtain $AV M_B = BV$, where $M_B = R_A^{-1} R_B$ is also triangular, with inverted eigenvalues on the diagonal.

However, the above treatment involves inversions of $R_A$ or $R_B$, and consequently could cause numerical instabilities whenever the triangular matrices to be inverted have very small diagonal entries. Moreover, in the presence of both zero and infinite eigenvalues, 
$R_A$ and $R_B$ are singular and cannot be inverted at all. Therefore, it would be ideal to construct a $Q$-free Schur form for any regular $(A,B)$ that does not rely on the inverse of the triangular matrices $R_A$ and $R_B$.

In the rest of the section, we show that such a ``$Q$-free'' partial Schur form is possible. 
We start with the following lemma.
\begin{lemma}\label{lem:triu}
Let $R_A$ and $R_B$ be upper triangular matrices, such that $|R_A(j,j)| + |R_B(j,j)| \neq 0$ for all $1\leq j\leq k$. Then for any upper triangular $G_1$ 
and $G_2$, and any~$\tau_j \in \mathbb{C}$, such that 
\eq{eq:G1}
G_1(j,j) = 
\left\{
\begin{array}{cl}
 \displaystyle  \tau_j, & |R_A(j,j)| < |R_B(j,j)| \\
 \displaystyle \frac{1- \tau_j R_B(j,j)}{R_A(j,j)}, & \mbox{otherwise} 
\end{array} 
\right.
\en
and 
\eq{eq:G2}
G_2(j,j) = 
\left\{
\begin{array}{cl}
 \displaystyle  \frac{1 - \tau_j R_A(j,j)}{R_B(j,j)}, & |R_A(j,j)| < |R_B(j,j)| \\
 \displaystyle \tau_j, &  \mbox{otherwise} 
\end{array} ;
\right.
\en
the matrix $G = R_A G_1 + R_B G_2$ is upper triangular with $G(j,j) = 1$ for all $1 \leq j\leq k$.
\end{lemma}
\begin{proof}
The matrix $G$ is a combination of upper triangular matrices and, hence, is upper triangular.  
For each diagonal entry of $G$, we have $G(j,j) = R_A(j,j) G_1(j,j) + R_B(j,j) G_2(j,j)$, where $R_A(j,j)$ and $R_B(j,j)$ do not
equal zero simultaneously. 
It is then easy to check that for any $\tau_j \in \mathbb{C}$, \eqref{eq:G1} and~\eqref{eq:G2} satisfy $R_A(j,j) G_1(j,j) + R_B(j,j) G_2(j,j) = 1$ for all $1\leq j\leq k$. 
\end{proof}

The following theorem gives the existence of a ``$Q$-free'' generalized Schur form. 
\begin{theorem}\label{thm:gschur2}
For any regular pair $(A,B)$ there exists a matrix $V \in \mathbb{C}^{n \times k}$ with orthonormal columns and upper triangular matrices 
$M_A, M_B \in \mathbb{C}^{k \times k}$, such that
\eq{eq:gschur2}
AV M_B = BV M_A,
\en
and $\lambda_j = M_A(j,j)/M_B(j,j)$ are (possibly infinite) eigenvalues of $(A,B)$.
\end{theorem}
\begin{proof}
We start from the generalized Schur decomposition~\eqref{eq:gschur}. Let $G_1$ and $G_2$ be upper triangular with diagonal elements 
defined by~\eqref{eq:G1} and~\eqref{eq:G2}. Right-multiplying both sides of the top and bottom equalities in~\eqref{eq:gschur} by $G_1$ and $G_2$, respectively,
and then summing up the results, gives an equivalent system
\eq{eq:gschur_tmp}
\left\{
\begin{array}{rcl}
A V & = & Q R_A, \\
AV G_1 + BV G_2 & = & Q G, 
\end{array} 
\right.
\en
where $$G = R_A G_1 + R_B G_2.$$ Since $(A,B)$ is a regular pair, $R_A(j,j)$ and $R_B(j,j)$ do not equal zero simultaneously for the same $j$. Thus, by Lemma~\ref{lem:triu},
$G$ is upper triangular with $G(j,j) = 1$ for all $j$. In particular, the latter implies that the inverse of $G$ exists. Hence, we can eliminate $Q$ from~\eqref{eq:gschur_tmp}.

It follows from the bottom equality in~\eqref{eq:gschur_tmp} that $Q = (AV G_1 + BV G_2)G^{-1}$. 
Substituting it to the top identity gives $AV = (AV G_1 + BV G_2)G^{-1} R_A$,
which implies~\eqref{eq:gschur2}~with 
\eq{eq:M}
M_A = G_2 G^{-1}R_A, \quad M_B = I - G_1 G^{-1} R_A . 
\en 
Furthermore, since the diagonal entries of $G$ are all ones, $G^{-1}(j,j) = 1$
for all $j$. Thus, from~\eqref{eq:M}, the diagonal elements of $M_A$ and $M_B$ 
are related to those of $R_A$ and $R_B$ by $M_A(j,j) = G_2(j,j) R_A(j,j)$ and 
$M_B(j,j) = 1 - G_1(j,j) R_A(j,j)$. 
In particular,~using definition~\eqref{eq:G1} and~\eqref{eq:G2} of $G_1(j,j)$ and $G_2(j,j)$,
if $|R_A(j,j)| < |R_B(j,j)|$, we obtain
\eq{eq:Mj1}
M_A(j,j) = \frac{1-\tau_j R_A(j,j)}{R_B(j,j)} R_A(j,j), \quad M_B(j,j) = 1 - \tau_j R_A(j,j),
\en
which implies that $M_A(j,j)/M_B(j,j) =  R_A(j,j)/R_B(j,j) \equiv \lambda_j$, 
provided that $\tau_j$ is chosen such that $1-R_A(j,j)\tau_j \neq 0$.

Similarly, if $|R_A(j,j)| \geq |R_B(j,j)|$, 
\eq{eq:Mj2}
M_A(j,j) = \tau_j R_A(j,j), \quad M_B(j,j) = \tau_j R_B(j,j).
\en
Hence, by choosing $\tau_j \neq 0$, we get $M_A(j,j)/M_B(j,j) =  R_A(j,j)/R_B(j,j) \equiv \lambda_j$.
\end{proof}


Given the traditional generalized Schur decomposition~\eqref{eq:gschur}, Theorem~\ref{thm:gschur2} suggests a simple approach
for obtaining the ``Q-free'' form~\eqref{eq:gschur2}. Namely, one starts with setting up the upper triangular matrices $G_1$, $G_2$,
and $G$ defined according to Lemma~\ref{lem:triu}, and then applies~\eqref{eq:M} to evaluate $M_A$ and $M_B$. The right Schur basis $V$ is 
the same in~\eqref{eq:gschur} and~\eqref{eq:gschur2}.      
%
Note that the definition~\eqref{eq:M} of the factors $M_A$ and~$M_B$ does not rely on inverting $R_A$ or $R_B$ and only requires an inverse 
of a triangular matrix $G$ which has all ones on the diagonal. 
Hence, 
it avoids inverting (nearly) singular triangular matrices.
Additionally, the proposed formulas~\eqref{eq:G1} and~\eqref{eq:G2}, as well as~\eqref{eq:Mj1},  
always have the largest modulus number between $R_A(j,j)$ and $R_B(j,j)$ in the denominator, which 
mitigates 
potential numerical issues introduced by the small diagonal elements of $R_A$ and $R_B$.

Clearly, the choice of $G_1$ and $G_2$ in~\eqref{eq:G1} and~\eqref{eq:G2}, and hence
decomposition~\eqref{eq:gschur2}, is not unique, as it depends on the value of the parameter $\tau_j$ and allows freedom in the choice
of entries above diagonal in $G_1$ and $G_2$. 
Therefore, in order to fix particular $M_A$ and $M_B$, in practice, we choose $G_1$ and $G_2$ to be diagonal and set
$\tau_j = 0$ if $|R_A(j,j)| < |R_B(j,j)|$ and $\tau_j = 1$, otherwise. This yields the diagonal matrices 
\eq{eq:G1_fix}
G_1(j,j) = 
\left\{
\begin{array}{cl}
 \displaystyle  0, & |R_A(j,j)| < |R_B(j,j)| \\
 \displaystyle \frac{1 - R_B(j,j)}{R_A(j,j)}, & \mbox{otherwise} 
\end{array} 
\right.
\en
and 
\eq{eq:G2_fix}
G_2(j,j) = 
\left\{
\begin{array}{cl}
 \displaystyle  1/R_B(j,j) , & |R_A(j,j)| < |R_B(j,j)| \\
 \displaystyle 1, &  \mbox{otherwise} 
\end{array} ;
\right.
\en
for which, by~\eqref{eq:Mj1}, $M_A(j,j) = R_A(j,j)/R_B(j,j)$, $M_B(j,j) = 1$ if $|R_A(j,j)| < |R_B(j,j)|$, and, by~\eqref{eq:Mj2}, 
$M_A(j,j) = R_A(j,j)$, $M_B(j,j) = R_B(j,j)$, otherwise.
Note that it is also possible to choose $\tau_j$ that give $M_A$ and $M_B$, such that $|M_A(j,j)|^2 + |M_B(j,j)|^2 = 1$, but our implementation follows the simple formulas~\eqref{eq:G1_fix} and~\eqref{eq:G2_fix}.

\section{The GPLHR algorithm for computing a partial Schur form}\label{sec:pqr}

For \textit{Hermitian} problems, a class of powerful eigensolvers is based on preconditioned block iterations; e.g.,~\cite{Knyazev:01, Quillen.Ye:10, Ve.Kn:14TR}.
Such methods fit into a unified framework consisting of~two key ingredients: generation of the trial (or search) subspace
and extraction of the approximate eigenvectors. Namely, given a number of eigenvector approximations
and a \textit{preconditioner}, 
at each step $i$, 
a low-dimensional trial subspace $\mathcal{Z}^{(i)}$ is constructed.
If properly chosen, $\mathcal{Z}^{(i)}$ contains improved approximations to 
the wanted eigenvectors, which are extracted from 
the subspace using a suitable projection procedure and are then used to start a new 
iteration. In this section, we extend this framework to the non-Hermitian case.

The results of Section~\ref{subsec:gschur_new} are central to our derivations. In particular, instead of the standard partial generalized Schur decomposition~\eqref{eq:gschur},
we focus on the ``$Q$-free'' form~\eqref{eq:gschur2}, and seek approximation of the right Schur vectors $V$ and the associated triangular factors $M_A$ and $M_B$.
As we shall see, the approximation to the left Schur vectors $Q$, as well as to triangular 
matrices $R_A$ and $R_B$, 
can be easily obtained as a by product of 
the proposed scheme. 

\subsection{Construction of the trial subspace}\label{subsec:trial_gschur}
Let $V^{(i)}$, $M_A^{(i)}$, and $M_B^{(i)}$ be approximations to 
$V$, $M_A$, and $M_B$ in~\eqref{subsec:gschur_new} at an iteration~$i$, such that $V^{(i)}$ has $k$ orthonormal columns, 
and $M_A^{(i)}$ and $M_B^{(i)}$ are $k$-by-$k$ upper triangular. 
In order to define a trial subspace that can be used for searching a new approximation $V^{(i+1)}$, 
we adopt the following subspace orthogonal correction viewpoint.
%

We are interested in constructing a subspace $\mathcal{K}^{(i)}$ that provides a good representation of the  
correction $C^{(i)}$, 
such that 
\eq{eq:schur_cor}
A(V^{(i)} + C^{(i)}) (M_B^{(i)} + \Delta_B^{(i)})  = B (V^{(i)} + C^{(i)}) (M_A^{(i)} + \Delta_A^{(i)}), 
\en 
where $V^{(i)*}C^{(i)} = 0$; and $\Delta_A^{(i)}$, $\Delta_B^{(i)}$ are corrections of the triangular factors $M_A^{(i)}$ and $M_B^{(i)}$, 
respectively. Once the ``correction subspace'' $\mathcal{K}^{(i)}$ is determined, the trial subspace can be defined as the subspace sum $\mathcal{Z}^{(i)} = \mbox{col} \{ V^{(i)} \} + \mathcal{K}^{(i)}$,
where $\mbox{col} \{ V^{(i)} \}$ denotes the column space of $V^{(i)}$. 

After 
rearranging the terms in~\eqref{eq:schur_cor}, 
we get 
\[
AC^{(i)} M_B^{(i)} - B C^{(i)} M_A^{(i)} = - (A V^{(i)} M_B^{(i)} - B V^{(i)} M_A^{(i)}) + BV \Delta_A^{(i)} - AV \Delta_B^{(i)}, 
\]
where $V = V^{(i)} + C^{(i)}$ is the matrix of the exact right Schur vectors.
Then, using~\eqref{eq:gschur}, 
we arrive at the identity 
\eq{eq:schur_cor1}
AC^{(i)} M_B^{(i)} - B C^{(i)} M_A^{(i)} = - (A V^{(i)} M_B^{(i)} - B V^{(i)} M_A^{(i)}) + Q(R_B \Delta_A^{(i)} - R_A \Delta_B^{(i)}), 
\en
where $Q$ is the matrix of the exact left Schur vectors.
Let $P_{Q^{\perp}} = I  - Q^{(i)} Q^{(i)*}$ be an orthogonal projector onto $\mbox{col} \{ Q^{(i)} \}^{\perp}$, where 
$Q^{(i)}$ is a matrix whose orthonormal columns represent a basis of $\mbox{col} \{ \delta_B AV^{(i)} + \delta_A BV^{(i)} \}$ for
some scalars $\delta_A, \delta_B \in \mathbb{C}$, such that $| \delta_A | + | \delta_B | \neq 0$.    
Then, applying $P_{Q^{\perp}}$ to both sides of~\eqref{eq:schur_cor1} and neglecting the high-order term 
$(I  - Q^{(i)} Q^{(i)*}) Q(R_B \Delta_A^{(i)} - R_A \Delta_B^{(i)})$
gives 
the equation
\eq{eq:sylvester}
(P_{Q^{\perp}} A P_{V^{\perp}})  C M_B^{(i)} - (P_{Q^{\perp}} B P_{V^{\perp}} ) C M_A^{(i)} =  - P_{Q^{\perp}}  (A V^{(i)} M_B^{(i)} - B V^{(i)} M_A^{(i)} ), 
\en
whose solution $C$, constrained to satisfy $V^{(i)*}C = 0$, provides an approximation to the desired correction $C^{(i)}$ of~\eqref{eq:schur_cor}. 
Here,  $P_{V^{\perp}} = I  - V^{(i)} V^{(i)*}$ is an orthogonal projector onto $\mbox{col} \{ V^{(i)} \}^{\perp}$ and, hence,
$C = P_{V^{\perp}} C$.

Equation~\eqref{eq:sylvester} represents a \textit{generalized Sylvester equation}~\cite{simoncini_matrix_eq:14TR} and
can be viewed as a block extension of the standard single vector Jacobi--Davidson (JD) correction equation~\cite{Fokkema:Sleijpen.Vorst:98, Sleijpen.Vorst:96},
where the right-hand side is given by the projected residual of problem~\eqref{eq:gschur2}. 
Note that a connection between 
the subspace correction  
and solution of a Sylvester equation was also observed in~\cite{Philippe.Saad:07}, 
though not in the context of the generalized Schur form computation. 
Thus, a possible option is to (approximately) solve~\eqref{eq:sylvester}
for a block correction $C$ and then set $\mathcal{K}^{(i)} = \mbox{col} \{ C \}$.

However, 
with this approach,
it is not straightforward to solve the matrix equation~\eqref{eq:sylvester} efficiently. Under certain assumptions, one can reduce~\eqref{eq:sylvester} to the \textit{standard} Sylvester equation, but this is generally expensive for large problems.

Instead, we do not seek the solution of the generalized Sylvester equation, 
and use a different strategy to generate the correction subspace $\mathcal{K}^{(i)}$. 
To this end, we consider~\eqref{eq:sylvester} as an equation of the general form   
\eq{eq:operator}
L(C) = F, \quad V^{(i)*}C = 0,
\en 
where $L(C) \equiv (P_{Q^{\perp}} A  P_{V^{\perp}} ) C M_B^{(i)} -  (P_{Q^{\perp}} B  P_{V^{\perp}})  C M_A^{(i)}$ and 
$F \equiv -  P_{Q^{\perp}} (A V^{(i)} M_B^{(i)} - B V^{(i)} M_A^{(i)})$.
Since $L(\cdot)$ is a linear operator, standard subspace projection techniques for solving linear systems, e.g.,~\cite{Greenbaum:97, Saad:03}, 
can be formally applied to~\eqref{eq:operator}. 
Thus, 
we can define a preconditioned Krylov-like subspace for~\eqref{eq:operator}. 
However,
instead of using this subspace to solve~\eqref{eq:operator}, 
we place it directly in an eigensolver's trial subspace.



More precisely, let $T_L: \mbox{col}\{Q^{(i)}\}^{\perp} \rightarrow \mbox{col}\{V^{(i)}\}^{\perp}$ be a \textit{preconditioner} for the operator $L(\cdot)$. 
Then, following the analogy with the Krylov subspace methods for (block) linear systems~\cite{Greenbaum:97, Gutknecht:07, Saad:03}, 
one can expect that the solution $C$ of~\eqref{eq:operator} can be well represented in the \textit{preconditioned block Krylov subspace}
\[
\mbox{block\:span}\left\{T_L(F), T_L\left(L(T_L(F)), \ldots, (T_LL)^m(T_L(F))\right)\right\}, 
\]
generated by the operator $T_L(L(\cdot))$ and the starting block $T_L(F) \in \mathbb{C}^{n \times k}$,
which contains all blocks of the form $\sum_{\ell =0}^m (T_LL)^\ell (T_L(F))G_\ell$, 
with $G_\ell \in \mathbb{C}^{k \times k}$~\cite{Gutknecht:07}. In particular, this implies that each column of $C$ should be 
searched in 
\eq{eq:krylov_sep_block} 
\mathcal{K}^{(i)} \equiv \mathcal{K}_{m+1}^{(i)} = 
\text{col} \{ W^{(i)}, S_1^{(i)}, S_2^{(i)}, \ldots, S_m^{(i)} \} , 
\en
where the blocks $K = [W^{(i)}, S_1^{(i)}, S_2^{(i)}, \ldots, S_m^{(i)}]$ represent what we call the Krylov--Arnoldi~sequence, generated
by a preconditioned Arnoldi-like procedure for problem~\eqref{eq:operator}, or equivalently, 
for the generalized Sylvester equation~\eqref{eq:sylvester}.
This procedure is stated in Algorithm~\ref{alg:barnoldi}, where the input parameter $m$ determines the~subspace size.
\begin{algorithm}[!htbp]
\begin{small}
\begin{center}
  \begin{minipage}{5in}
\begin{tabular}{p{0.5in}p{4.5in}}
{\bf Input}:  &  \begin{minipage}[t]{4.0in}
                  preconditioning operator $T_L$, the parameter $m$.
                  \end{minipage} \\
{\bf Output}:  &  \begin{minipage}[t]{4.0in}
                  the Krylov--Arnoldi basis $K = [W, S_1, \ldots, S_m]$ 
                  \end{minipage}
\end{tabular}
\begin{algorithmic}[1]
\STATE $l \gets 0$; $W \gets \mbox{orth}\footnote{Throughout, $\mbox{orth}(V)$ denotes a procedure for orthonormalizing columns of $V$.}\left(T_L(F)\right)$; 
\FOR {$l = 1 \rightarrow m$}
     \STATE $S_{l} \gets T_L ( L (S_{l-1})) $; $S_l \gets S_l -  W( W^*S_l)$; 
     \FOR {$j = 1 \rightarrow l-1$}
         \STATE $S_l \gets S_l - S_j(S_j^*S_l)$; 
     \ENDFOR
     \STATE $S_l \gets \mbox{orth}(S_l)$; 
\ENDFOR
\STATE Return $K = [W, S_1, \ldots, S_m]$.
\end{algorithmic}
\end{minipage}
\end{center}
\end{small}
  \caption{Preconditioned block Arnoldi type procedure for equation~\eqref{eq:operator}}
  \label{alg:barnoldi}
\end{algorithm}

Thus, given the subspace~\eqref{eq:krylov_sep_block}, capturing the orthogonal correction $C$, the eigensolver's trial subspace $\mathcal{Z}^{(i)}$ can be 
defined as $\mbox{col}\{V^{(i)}\} + \mathcal{K}_{m+1}^{(i)}$,
i.e.,
\eq{eq:trial_partial}
\mathcal{Z}^{(i)} = \text{col} \{ V^{(i)}, W^{(i)}, S_1^{(i)}, S_2^{(i)}, \ldots, S_m^{(i)} \}
\en
Clearly, the Krylov--Arnoldi sequence vectors $K$, generated by Algorithm~\ref{alg:barnoldi}, represents a system of    
orthonormal columns 
that are orthogonal to $V^{(i)}$
due to the choice of the preconditioner $T_L: \mbox{col}\{Q^{(i)}\}^{\perp} \rightarrow \mbox{col}\{V^{(i)}\}^{\perp}$. 
Hence, the trial subspace $\mathcal{Z}^{(i)}$ in~\eqref{eq:trial_partial}
is spanned by orthonormal vectors, which gives a stable basis.

A desirable feature of the preconditioner $T_L$ for problem~\eqref{eq:operator} (or~\eqref{eq:sylvester}) is that it should 
approximate the inverse of the operator $L(\cdot)$. One way to construct such an approximation is to replace the upper triangular matrices $M_A^{(i)}$ and $M_B^{(i)}$ in the expression 
for $L(\cdot)$ by diagonals $\sigma_A I$ and $\sigma_B I$, respectively, where $\sigma_A/\sigma_B = \sigma$. 
This results in an approximation of $L(\cdot)$ that is given by the matrix 
$(I - Q^{(i)}Q^{(i)*})(\sigma_B A - \sigma_A B)(I - V^{(i)}V^{(i)*})$. Thus, one can choose the preconditioner as 
\eq{eq:prec_gen}
T_L = (I - V^{(i)} V^{(i)*}) T (I - Q^{(i)}Q^{(i)*}), 
\en
where $T \approx (\sigma_B A - \sigma_A B)^{\dagger}$. Throughout, we consider $T$ as a preconditioner for eigenvalue problem~\eqref{eq:gevp},
which should be distinguished from the preconditioner $T_L$ for the generalized Sylvester equation.
Note that, in practice, 
assuming that
$\sigma$ is a finite shift that is different from an eigenvalue, 
we can set $\sigma_A = \sigma$ and $\sigma_B = 1$, so that $T \approx (A - \sigma B)^{-1}$ is an approximate shift-and-invert operator.

Finally, given the definition~\eqref{eq:prec_gen} of the preconditioner $T_L$, we can derive explicit expressions for the blocks in the Krylov--Arnoldi 
sequence generated from Algorithm~\ref{alg:barnoldi}.
It is easy to check that 
\eq{eq:WA_gschur}
W^{(i)}  =  \mbox{orth} ( (I - V^{(i)}V^{(i)*}) T (I - Q^{(i)}Q^{(i)*}) (AV^{(i)}M_B^{(i)} - B V^{(i)} M_A^{(i)}) ),\\
\en
and, for $l = 1, \ldots, m$,
\eq{eq:SA_gschur}
\begin{array}{ccl}
\hat S_l^{(i)} & = & (I - V^{(i)}V^{(i)*}) T (I - Q^{(i)}Q^{(i)*}) (AS_{l-1}^{(i)} M_B^{(i)} - B S_{l-1}^{(i)} M_A^{(i)}), \\
S_l^{(i)} & = & \mbox{orth}(P^{(l-1)}_{S^{\perp}} (I - W^{(i)}W^{(i)*})  \hat S_l^{(i)}); 
\end{array}
\en
where $S_0^{(i)} \equiv W^{(i)}$ and $P^{(l-1)}_{S^{\perp}}$ is the orthogonal projector onto $\text{col}\{S_1^{(i)}, \ldots, S_{l-1}^{(i)}\}^{\perp}$,  
\eq{eq:projS}
P^{(l-1)}_{S^{\perp}} = (I - S_{l-1}^{(i)} S_{l-1}^{(i-1)*})(I - S_{l-2}^{(i)} S_{l-2}^{(i-1)*}) \cdots (I - S_{1}^{(i)} S_{1}^{(i-1)*}).
\en

We remark that the above construction of the trial subspace~\eqref{eq:trial_partial} is similar in spirit to 
the one used to devise the BIFP Krylov subspace methods in~\cite{Golub.Ye:02, Quillen.Ye:10} for computing extreme eigenpairs of Hermitian problems.
The conceptual difference, however, is that we put the block correction equation~\eqref{eq:sylvester} in the center stage, 
whereas the BIFP methods are based on Krylov(-like) subspaces delivered by a simultaneous solution of problems $(A - \sigma_j B) w = r$ 
for $k$ different shifts $\sigma_j$. 
In particular, this allows us to discover
the projectors $I - V^{(i)}V^{(i)*}$ and $I - Q^{(i)}Q^{(i)*}$
in~\eqref{eq:prec_gen}.
These projectors are then blended into formulas~\eqref{eq:WA_gschur}--\eqref{eq:projS} 
that define the trial subspace. Their presence 
has a strong effect on eigensolver's robustness, as we shall demonstrate in the numerical experiments of Section~\ref{sec:num}.



%

\subsection{The trial subspace for standard eigenvalue problem}\label{subsec:trial_stand}
If $B = I$, 
instead of the generalized form~\eqref{eq:gschur},
we seek the standard partial Schur decomposition 
\eq{eq:schur}
A V = V R,
\en 
where $V \in \IC^{n \times k}$ contains orthonormal Schur vectors and $R \in \IC^{k \times k}$ is upper triangular with 
the wanted eigenvalues of $A$ on its diagonal. It is clear that~\eqref{eq:schur} is a special case of the ``$Q$-free'' form~\eqref{eq:gschur2}
with $B = I$, $M_A = R$, and $M_B = I$. Therefore, the derivation of the trial subspace, described in the previous section, is directly applicable here. 

In particular, let $V^{(i)}$ be an approximate Schur basis, and let $M_A^{(i)}$ and $M_B^{(i)}$ be 
the associated upper triangular matrices that 
approximate $R$ and $I$, respectively.       
Then, following the arguments of Section~\ref{subsec:trial_gschur}, from~\eqref{eq:schur_cor} with $B = I$,
we get 
\eq{eq:schur_cor1_stand}
AC^{(i)} M_B^{(i)} - C^{(i)} M_A^{(i)} = - (A V^{(i)} M_B^{(i)} - V^{(i)} M_A^{(i)}) + V(\Delta_A^{(i)} - R \Delta_B^{(i)}). 
\en 
This equality is the analogue to~\eqref{eq:schur_cor1} for a standard eigenvalue problem. 
Here, in order to approximate the correction $C^{(i)}$, it is natural to apply the projector 
$P_{V^{\perp}} = I - V^{(i)}V^{(i)*}$ to both sides of~\eqref{eq:schur_cor1_stand} and neglect the term
$(I - V^{(i)}V^{(i)*})V(\Delta_A^{(i)} - R \Delta_B^{(i)})$. As a result, we obtain the equation    
\eq{eq:sylvester_st}
(P_{V^{\perp}}AP_{V^{\perp}}) C M_B^{(i)} - C M_A^{(i)} =  - P_{V^{\perp}} (A V^{(i)} M_B^{(i)} - V^{(i)} M_A^{(i)} ), 
\en
where the solution $C$, which is orthogonal to $V^{(i)}$, approximates the desired $C^{(i)}$. 
Note that~in contrast to the correction equation~\eqref{eq:sylvester}
for the generalized eigenvalue problem, which is based on two projectors $P_{V^{\perp}}$ and $P_{Q^{\perp}}$,~\eqref{eq:sylvester_st} 
involves only one projector~$P_{V^{\perp}}$. This is expected, as both $V^{(i)}$ and $Q^{(i)}$approximate the same
vectors,~$V$.    
 

Considering~\eqref{eq:sylvester_st} as an equation of the general form~\eqref{eq:operator}, 
where $L(C) = (P_{V^{\perp}} A P_{V^{\perp}}) C M_B^{(i)} - C M_A^{(i)}$ and 
$F = - P_{V^{\perp}} (A V^{(i)} M_B^{(i)} - V^{(i)} M_A^{(i)} )$,
we apply $m$ steps of Algorithm~\ref{alg:barnoldi} with the preconditioner
\eq{eq:prec}
T_L = (I - V^{(i)} V^{(i)*}) T (I - V^{(i)}V^{(i)*}). 
\en
This yields 
the trial subspace~\eqref{eq:trial_partial},
where
\eq{eq:WA_schur}
W^{(i)}  =  \mbox{orth} ( (I - V^{(i)}V^{(i)*}) T (I - V^{(i)}V^{(i)*}) (AV^{(i)} M_B^{(i)} - V^{(i)} M_A^{(i)}) ),\\
\en
and, for $l = 1, \ldots, m$,
\eq{eq:SA_schur}
\begin{array}{ccl}
\hat S_l^{(i)} & = & (I - V^{(i)}V^{(i)*}) T (I - V^{(i)}V^{(i)*}) (AS_{l-1}^{(i)} M_B^{(i)} - S_{l-1}^{(i)} M_A^{(i)}), \\
S_l^{(i)} & = & \mbox{orth}(P^{(l-1)}_{S^{\perp}} (I - W^{(i)}W^{(i)*})  \hat S_l^{(i)}). 
\end{array}
\en
As before, we assume that $S_0^{(i)} \equiv W^{(i)}$ in~\eqref{eq:SA_schur} and that 
$P^{(l-1)}_{S^{\perp}}$ is the orthogonal projector 
defined in~\eqref{eq:projS}; the preconditioner $T$ in~\eqref{eq:prec} 
is chosen as an approximation of the shift-and-invert operator $(A - \sigma I)\inv$.

\subsection{The harmonic Schur--Rayleigh--Ritz procedure}\label{subsec:SRR}

Let $\mathcal{Z}^{(i)}$ be a trial subspace of dimension $s = (m+2)k$ at iteration $i$ 
defined by~\eqref{eq:trial_partial}
with~\eqref{eq:WA_gschur}--\eqref{eq:projS}. 
We try to find a new orthonormal approximate 
Schur basis $V^{(i+1)}$ and the corresponding $k$-by-$k$ upper triangular matrices $M_A^{(i+1)}$ and $M_B^{(i+1)}$, 
such that each column of $V^{(i+1)}$ belongs to~$\mathcal{Z}^{(i)}$ and
$M_A^{(i+1)}, M_B^{(i+1)}$ approximate the triangular factors in~\eqref{eq:gschur2}. 

To fulfill this task, we use a 
harmonic Rayleigh--Ritz projection~\cite{Morgan:91, Morgan.Zeng:98}, adapted to the case 
of the ``Q-free'' 
Schur form~\eqref{eq:gschur2}. 
Specifically, we let $\mathcal{U}^{(i)} = ( A - \sigma B ) \mathcal{Z}^{(i)}$
be a \textit{test subspace}, where we assume that the target shift $\sigma$ 
is not an eigenvalue.
Then the 
approximations $V^{(i+1)}$, $M_A^{(i+1)}$, and $M_B^{(i+1)}$ 
can be determined by 
requiring 
each column of the residual of~\eqref{eq:gschur2}
to be orthogonal to this test subspace, i.e., 
\begin{equation}\label{eq:pg}
AV^{(i+1)} M_B^{(i+1)} - B V^{(i+1)}  M_A^{(i+1)} \perp \mathcal{U}^{(i)}.
\end{equation}   
%
%
This yields the projected problem
\begin{equation}\label{eq:proj_schur}
(U^* A Z) Y M_B^{(i+1)} = (U^* B Z) Y M_A^{(i+1)},
\end{equation}
where $Y \in \IC^{s \times k}$, $M_A^{(i+1)}, M_B^{(i+1)} \in \mathbb{C}^{k \times k}$ are unknown; and $Z, U  \in \mathbb{C}^{n \times s}$ contain orthonormal columns spanning 
$\mathcal{Z}^{(i)}$ and $\mathcal{U}^{(i)}$, respectively. Once we have computed $Y$, $M_A^{(i+1)}$, and $M_B^{(i+1)}$, 
approximations to the Schur vectors, determined by~\eqref{eq:pg}, are given by 
$V^{(i+1)} = Z Y$.
%
%
%

Let us now consider the solution of the projected problem~\eqref{eq:proj_schur}. We first show
that the projected matrix pair $(U^* A Z, U^* B Z)$ is regular. 
\begin{proposition}\label{prop:projp}
Let $Z, \,U = \text{\emph{orth}}((A - \sigma B) Z) \in \mathbb{C}^{n \times s}$ contain orthonormal basis vectors of the trial and test subspaces,
and 
assume that $\sigma$ is different from any eigenvalue of $(A, B)$. Then 
the projected pair $(U^* A Z, U^* B Z)$ is regular.     
\end{proposition}
\begin{proof}
Assume, on the contrary, that $(U^* A Z, U^* B Z)$ is singular. Then, by definition of a singular matrix pair,
regardless of the choice of $\sigma$, 
\eq{eq:det0}
\mbox{det}(U^* A Z - \sigma U^* B Z) = 0.
\en
Since $U$ is an orthonormal basis of $\mbox{col}\{ (A - \sigma B) Z \}$, we can write
$(A - \sigma B) Z  = UF$, where $F$ is an $s$-by-$s$ square matrix. 
This matrix is nonsingular, i.e., $\mbox{det}(F) \neq 0$, because $A - \sigma B$ is nonsingular. 
Hence, $U = (A - \sigma B) Z F\inv$, and
\begin{eqnarray*}
U^* A Z - \sigma U^* B Z  & = & F^{-*} Z^* (A - \sigma B)^* A Z - \sigma F^{-*} Z^* (A - \sigma B)^* B Z \\
& = &  F^{-*} Z^* (A - \sigma B)^* (A - \sigma B) Z.
\end{eqnarray*}
Thus, from the above equality and~\eqref{eq:det0}, we have
\[
0= \mbox{det}(U^* A Z - \sigma U^* B Z) = \mbox{det}( F^{-*}) \mbox{det}( Z^* (A - \sigma B)^* (A - \sigma B) Z ).
\]
Since $\mbox{det}(F) \neq 0$, this identity implies that $\mbox{det}( Z^* (A - \sigma B)^* (A - \sigma B) Z ) = 0$.
But the matrix $Z^* (A - \sigma B)^* (A - \sigma B) Z$ is Hermitian positive semidefinite and its singularity
implies that $\mbox{det}(A - \sigma B) = 0$, which contradicts the assumption that 
$\sigma \notin \Lambda(A,B)$.
\end{proof}

Thus, problem~\eqref{eq:proj_schur} is of the form~\eqref{eq:gschur2} and the matrix pair $(U^* A Z, U^* B Z)$ is regular, 
whenever
$\sigma$ is not an eigenvalue of $(A, B)$.
Therefore,~\eqref{eq:proj_schur} 
can be solved by the approach
described in Section~\ref{subsec:gschur_new}. 

Specifically, 
one first employs the standard sorted QZ algorithm~\cite{Moler.Stewart:73} to compute the 
\text{full} generalized Schur form of $(U^* A Z, U^* B Z)$. This will produce two unitary matrices $\bar Y_L, \bar Y_R \in \mathbb{C}^{s \times s}$ of the left and right generalized Schur
vectors, 
along with the upper triangular $\bar R_A, \bar R_B  \in \mathbb{C}^{s \times s}$, 
ordered in such a way that the ratios $\bar R_1(j,j)/\bar R_2(j,j)$ of the first 
$k$ diagonal elements 
are closest to $\sigma$, and the corresponding Schur vectors are located at the $k$ leading columns of $\bar Y_L$ and $\bar Y_R$.
Then, we let $Y_L = \bar Y_L(\mbox{\texttt{:,1:k}}), Y_R = \bar Y_R(\mbox{\texttt{:,1:k}}) \in \mathbb{C}^{s \times k}$, and obtain the desired Schur basis in~\eqref{eq:proj_schur} as $Y = Y_R$.  
The triangular factors are given as $M_A^{(i+1)} = G_2 G^{-1}\tilde R_A$ and $M_B^{(i+1)} = I - G_1 G^{-1} \tilde R_A$, where 
$G = \tilde R_A G_1 + \tilde R_B G_2$,
$\tilde R_A \equiv \bar R_A(\texttt{1:k,1:k})$ and $\tilde R_B \equiv \bar R_B(\texttt{1:k,1:k})$, and $G_1$, $G_2$ are diagonal matrices defined 
by~\eqref{eq:G1_fix} and~\eqref{eq:G2_fix} with $R_A$, $R_B$ replaced with $\tilde R_A$, $\tilde R_B$, respectively.

As a result, the described extraction approach, which we call a harmonic Schur--Rayleigh--Ritz (SRR) procedure,
constructs a new basis $V^{(i+1)} = Z Y$ of the approximate Schur vectors (the harmonic Schur--Ritz vectors)  
and the upper triangular matrices~$M_A^{(i+1)}$, $M_B^{(i+1)}$, such that 
the ratios of their diagonal entries are approximations to the desired eigenvalues.
Note that in the case of a standard eigenvalue problem, the proposed use of formulas~\eqref{eq:G1_fix} and~\eqref{eq:G2_fix} 
for evaluating the matrices $M_A^{(i+1)}$ and $M_B^{(i+1)}$
indeed guarantees that they converge to $R$ and $I$, as 
$\tilde R_A$ and $\tilde R_B$ get closer to $R$ and $I$, respectively.    

Although both the construction of the trial subspace and the presented harmonic SRR procedure
aim at finding the factors $V$, $M_A$, and $M_B$ of the ``Q-free'' form~\eqref{eq:gschur2}, it is easy to see that, as a by-product, 
the scheme can also produce approximations 
to the left Schur vectors $Q$ and the upper triangular factors $R_A$, $R_B$ of the conventional generalized
Schur decomposition~\eqref{eq:gschur} if $B \neq I$. Specifically, the former is given by $Q^{(i+1)} = U Y_L$, whereas the latter
correspond to $\tilde R_A$ and $\tilde R_B$. 
%


Note that the computed $Q^{(i+1)}$ can be conveniently exploited at iteration 
$i+1$ to construct the new test subspace $\mathcal{U}^{(i+1)} = (A - \sigma B) \mathcal{Z}^{(i+1)}$ and, if $B \neq I$, 
define the projector~$P_{Q^{\perp}}$, because $Q^{(i+1)}$
%
represents an orthonormal basis of $\mbox{col} \{ (A - \sigma B) V^{(i+1)} \}$.


\begin{proposition}\label{prop:U}
Let $Z, \,U = \text{orth}((A - \sigma B) Z) \in \mathbb{C}^{n \times s}$ contain orthonormal basis vectors of the trial and test subspaces. 
Assume that $V = ZY_R, \,Q = U Y_L \in \mathbb{C}^{n \times k}$ are the approximate right and left generalized Schur vectors resulting from the harmonic SRR procedure,
along with the upper triangular factors \textup{$\tilde R_A \equiv \bar R_A(\texttt{1:k,1:k})$} and \textup{$\tilde R_B \equiv  \bar R_B(\texttt{1:k,1:k})$}. 
Then,
\[
(A - \sigma B) V = Q (\tilde R_A - \sigma \tilde R_B),
\]  
i.e., $Q = U Y_{L}$ represents an orthonormal basis of $\text{\emph{col}}\{(A - \sigma B) V\}$.   
\end{proposition}
\begin{proof}
%
Since $UU^*$ is an orthogonal projector onto $\text{col}\{(A - \sigma B) Z\}$ and $V = ZY_R$, we have
\[
(A - \sigma B) V  = UU^*(A - \sigma B) Z Y_{R} = U (U^*AZ) Y_{R} - \sigma U(U^* B Z) Y_{R}.
\]
But $(U^*AZ) Y_{R} = Y_L \tilde R_A$ and $(U^* B Z) Y_{R} = Y_L \tilde R_B$ (resulting from the QZ factorization of $(U^*AZ,U^*BZ)$), which, combined with the equality $Q = U Y_{L}$, gives the desired result.
\end{proof}




Finally, we remark that another standard way to extract Schur vectors is to~choose the test subspace $\mathcal{U}^{(i)}$ to be the same as the trial subspace 
$\mathcal{Z}^{(i)}$.
This yields the~standard 
Rayleigh--Ritz type procedure.
%
%
However, the drawback of this approach
is that 
Schur approximations 
tend to be close to
\textit{exterior} eigenvalues and may not be successful for extracting the \textit{interior} eigenvalues and their Schur vectors~\cite{Fokkema:Sleijpen.Vorst:98, Morgan.Zeng:98}.
Additionally, it does not produce approximations of the left Schur vectors $Q^{(i+1)}$ that may be of interest and are needed to define the 
projector~$P_{Q^{\perp}}$. 
Therefore, in our algorithmic developments, we rely on the systematic use of the 
\textit{harmonic} SRR procedure.

\subsection{Augmenting the trial subspace}\label{subsec:conj_dir}
%

At every iteration $i$, the construction of the trial subspace $\mathcal{Z}^{(i)}$ starts with the 
current approximate Schur basis~$V^{(i)}$, which is used to determine the remaining blocks in~\eqref{eq:trial_partial}.
A natural question is whether it is possible to further enhance~\eqref{eq:trial_partial} 
if an additional information is 
available at the present stage of computation, without incurring any significant extra calculations. 
Namely, we are interested in defining a block $P^{(i)}$ of additional search directions, such that
the $(m+3)k$-dimensional subspace    
\eq{eq:trial_partial1}
\mathcal{Z}^{(i)} = \text{col} \{ V^{(i)}, W^{(i)}, S_1^{(i)}, \ldots, S_m^{(i)}, P^{(i)} \}
\en
contains better approximations to the desired Schur basis than~\eqref{eq:trial_partial}.

When solving Hermitian eigenvalue problems, a common technique to accelerate the eigensolver's convergence is to utilize an 
approximate solution computed at the previous step. The same idea can be readily applied to the non-Hermitian case.
For example, given approximate right generalized Schur vectors $V^{(i-1)}$ from the previous iteration, we can use them to 
expand the trial subspace~\eqref{eq:trial_partial} by setting $P^{(i)}$ to $V^{(i-1)}$ or, in a LOBPCG fashion, 
to a combination of $V^{(i-1)}$ and $V^{(i)}$, such that
\eq{eq:P_lobpcg}
P^{(i)} = V^{(i)} - V^{(i-1)} C_V^{(i-1)},
\en
where $C_V^{(i-1)}$ is an appropriate matrix~\cite{Knyazev:01, Kn.Ar.La.Ov:07}.
Clearly, in exact arithmetic, both options lead to the same subspace~\eqref{eq:trial_partial1}.  

Unlike the case of Hermitian eigenvalue problems, where utilizing approximation from the previous step 
leads to a significant acceleration of the eigensolver's convergence, our numerical experience suggests a different picture for non-Hermitian problems.
While in many cases the use of the additional search directions based on $V^{(i-1)}$ indeed lead to a noticeable improvement in convergence, 
we also observe situations in which no or only minor 
benefit can be seen.

A better approach for augmenting~\eqref{eq:trial_partial} is to define $P^{(i)}$ as a block 
of $k$ extra approximate Schur vectors. 
Specifically, if $\bar Y_R$ contains the full (ordered) right Schur
basis of the projected pair~$(U^* A Z , U^* B Z)$, where $U$ and $Z$ are orthonormal bases of the test and trial subspaces 
$\mathcal{U}^{(i-1)}$ and $\mathcal{Z}^{(i-1)}$, respectively, then we let 
\eq{eq:P1}
P^{(i)} \equiv V_{k+1:2k}^{(i)} = Z \bar Y_R(\texttt{:},\texttt{k+1:2k}),
\en
where $Y_R(\texttt{:},\texttt{k+1:2k})$ denotes a submatrix of $Y_R$ lying in
the columns $k+1$ through $2k$. 
As we demonstrate numerically in Section~\ref{sec:num}, this choice generally 
leads to 
a better
augmented subspace~\eqref{eq:trial_partial1}, 
significantly accelerating and stabilizing the convergence. 
Therefore, we use it by default in our implementation.

The use of formulation~\eqref{eq:P1} 
can be viewed as a 
\textit{``thick restart''}~\cite{Morgan:96, Sleijpen.Vorst:96, Stathopoulos.Saad.Wu:98, Wu.Simon:00}, 
which retains more (harmonic) Ritz approximations than needed 
after collapsing the trial subspace.
In particular, it suggests that GPLHR should transfer $2k$ approximate Schur vectors
from one iterations to another. These Schur vectors  
correspond to the $2k$ approximate eigenvalues closest to $\sigma$.
The~$k$ leading approximations then constitute the block $V^{(i)}$, 
whereas the $k$ remaining ones are placed in $P^{(i)}$.

\subsection{The GPLHR algorithm}\label{subsec:qz}

We are now ready to combine the developments of the previous sections 
and introduce 
the \textit{Generalized Preconditioned Locally Harmonic Residual (GPLHR) algorithm}.

Starting with an initial guess $V^{(0)}$, at each iteration $i$, GPLHR  
constructs a trial 
subspace~\eqref{eq:trial_partial1}
using the preconditioned Krylov--Arnoldi sequence~\eqref{eq:WA_gschur}--\eqref{eq:projS} if $B \neq I$, or~\eqref{eq:WA_schur}--\eqref{eq:SA_schur}
otherwise, and performs the harmonic SRR procedure 
to extract an updated set of Schur vector 
approximations with the corresponding upper triangular factors. 
The detailed description 
is summarized in Algorithm~\ref{alg:gplhr-qz}. 

\begin{algorithm}[!htbp]
\begin{small}
\begin{center}
  \begin{minipage}{5in}
\begin{tabular}{p{0.5in}p{4.5in}}
{\bf Input}:  &  \begin{minipage}[t]{4.0in}
                  A regular pair $(A, B)$ of $n$-by-$n$ matrices, shift $\sigma \in \IC$ different from any eigenvalue of $(A,B)$, 
                  preconditioner $T$, 
                  starting guess of the Schur vectors $V^{(0)} \in \IC^{n \times k}$, and the subspace expansion parameter $m$. 
                  \end{minipage} \\
{\bf Output}:  &  \begin{minipage}[t]{4.0in}
                  If $B \neq I$, then approximate generalized Schur vectors $V, Q \in \IC^{n \times k}$ and the 
                  associated upper triangular matrices $R_A, R_B \in \IC^{k \times k}$ in~\eqref{eq:gschur}, such that $\lambda_j = R_A(j,j)/R_B(j,j)$
                  are the $k$ eigenvalues of $(A,B)$ closest to $\sigma$. \\
                  If $B = I$, then the Schur vectors $V \in \IC^{n \times k}$ and the associated triangular matrix 
                  $R \in \IC^{k \times k}$ in~\eqref{eq:schur}, such that $\lambda_j = R(j,j)$
                  are the $k$ eigenvalues of $A$ closest to $\sigma$. 
                  \end{minipage}
\end{tabular}
\begin{algorithmic}[1]
\STATE $V \gets \texttt{orth}(V^{(0)})$;  $Q \gets \texttt{orth}((A - \sigma B) V)$; $P \gets [ \ ]$;
\STATE $[R_A,R_B,Y_L,Y_R] \gets \mbox{ordqz}(Q^*AV,Q^*BV,\sigma)$\footnote{$[R_A,R_B,Y_L,Y_R] \gets \mbox{ordqz}(\Phi,\Psi,\sigma)$
computes the full generalized Schur decomposition for an input pair $(\Phi,\Psi)$, ordered to ensure that 
$|R_A(i,i)/R_B(i,i) - \sigma| \leq |R_A(j,j)/R_B(j,j) - \sigma|$ for $i < j$.};
$V \gets V Y_R$; $Q \gets Q Y_L$; 
\STATE Set $k$-by-$k$ diagonal matrices $G_1$ and $G_2$
by~\eqref{eq:G1_fix} 
and~\eqref{eq:G2_fix};
$G \gets R_A G_1 + R_B G_2$.  
\STATE $M_A \gets G_2 G\inv R_A$; $M_B \gets I - G_1 G\inv R_A$; 
\WHILE {convergence not reached}
  \IF{$B \neq I$}
     \STATE $W_A \gets AV - Q R_A$; $W_B \gets BV - Q R_B$;
     \STATE $W \gets (I - VV^*)T(I - QQ^*)(W_A M_B - W_B M_A)$; 
  \ELSE
     \STATE $W \gets (I - VV^*)T(I - VV^*)(AV M_B - V M_A)$; 
  \ENDIF
  \STATE $W \gets \mbox{orth}(W)$; $S_0 \gets W$; $S \gets [ \ ]$;
  \FOR {$l = 1 \rightarrow m$}
     \IF{$B \neq I$}
        \STATE $S_{l} \gets (I-VV^*)T(I-QQ^*)(A S_{l-1} M_B - BS_{l-1} M_A)$;
     \ELSE
        \STATE $S_{l} \gets (I-VV^*)T(I-VV^*)(A S_{l-1} M_B - S_{l-1} M_A)$;
     \ENDIF
     \STATE $S_l \gets S_l -  W( W^*S_l)$; $S_l \gets S_l - S(S^*S_l)$; 
$S_l \gets \mbox{orth}(S_l)$; $S \gets [S \ S_l]$;
  \ENDFOR
  \STATE $P \gets P - V(V^*P)$; $P \gets P - W(W^*P)$; $P \gets P - S(S^*P)$; $P \gets \mbox{orth}(P)$;
  \STATE Set the trial subspace $Z \gets [V, \  W, \  S_{1}, \ldots,  \  S_{m}, \  P ]$;
  \STATE $\hat Q \gets \texttt{orth}( (A - \sigma B) [W, \ S_1, \ \ldots, \ S_m, P] )$; $\hat Q \gets \hat Q - Q(Q^*\hat Q)$; 
  \STATE Set the test subspace $U \gets [Q, \hat Q]$;
$[\bar R_A,\bar R_B,\bar Y_L, \bar Y_R] \gets \mbox{ordqz}(U^*AZ,U^*BZ, \sigma)$; 
  \STATE $Y_{R} \gets \bar Y_R(\texttt{:,1:k})$, $Y_{L} \gets \bar Y_L(\texttt{:,1:k})$, $R_{A} \gets \bar R_A(\texttt{1:k,1:k})$, $R_{B} \gets \bar R_B(\texttt{1:k,1:k})$; 
  \STATE  $P \gets W Y_W  +  S_{1} Y_{S_1} + \ldots +  S_{m} Y_{S_m} +  P Y_{P}$, where $Y_R \equiv [Y_V^T, \ Y_W^T, \ Y_{S_1}^T, \ \ldots, \ Y_{S_m}^T, Y_P^T]^T$ is a conforming 
partitioning of $Y_R$; 
  \STATE  $V \gets V Y_V + P$; $Q \gets U Y_{L}$; 
\STATE $P \gets Z \bar Y_R(\texttt{:,k+1:2k}))$;\footnote{This step can be disabled if the LOBPCG-style search direction of step 26 is preferred.} 
\STATE Set $k$-by-$k$ diagonal matrices $G_1$ and $G_2$ by~\eqref{eq:G1_fix} and~\eqref{eq:G2_fix}; $G \gets R_A G_1 + R_B G_2$;  
\STATE $M_A \gets G_2 G\inv R_A$; $M_B \gets I - G_1 G\inv R_A$; 
\ENDWHILE
\STATE If $B = I$, then $R \gets M_B\inv M_A$.
\end{algorithmic}
\end{minipage}
\end{center}
\end{small}
  \caption{The GPLHR algorithm}
  \label{alg:gplhr-qz}
\end{algorithm}

An attractive feature of Algorithm~\ref{alg:gplhr-qz} is that it is based only on block operations, which enables efficient 
level-3 BLAS routines for basic dense linear algebra, as well as utilization of sparse matrix times block (matblock) kernels. 
The method also provides a unified treatment 
of standard and generalized eigenproblems.
In the latter case, it also allows handling infinite eigenvalues if the 
chosen shift $\sigma$ is sufficiently large. 

Some caution should be exercised when choosing the value of the shift $\sigma$. 
In~particular, it is expected that $\sigma$ 
is different from any eigenvalue $\lambda$ of~\eqref{eq:gevp}. 
While the shift is unlikely to be an eigenvalue,
it is generally desirable that $\sigma$ is not very close~to~$\lambda$, which can potentially lead to ill-conditioning of the SRR
procedure and of the preconditioner $T \approx (A - \sigma B)\inv$. 
An optimal choice of $\sigma$ is outside the scope of this~paper.

In the case of a generalized eigenproblem,
a possible way to determine convergence 
is based on assessing column norms of the blocks $W_A$ and $W_B$ in step 7. These~blocks represent 
natural residuals of the generalized Schur form~\eqref{eq:gschur}. Hence, the convergence of $j$ initial columns 
$v_1, \ldots, v_j$ and $q_1, \ldots, q_j$ in $V$ and $Q$ can be declared if 
$\|W_A(\texttt{:,j})\|^2 + \|W_B(\texttt{:,j})\|^2$
is sufficiently small.
If $B = I$, a similar approach, based on examining the norms of the column of the residual $A V M_B - V M_A$, can be used.


One can observe that the block $W$,
formed in step 8 of the algorithm, equals $(I - VV^*)T(I - QQ^*)(AVM_B - BV M_A)$, 
because by definition of $W_A$ and $W_B$,
\begin{eqnarray}
\nonumber (I - QQ^*)(W_A M_B - W_B M_A) & = & (I - QQ^*)((AV - QR_A) M_B - (BV - QR_B) M_A)  \\ 
\nonumber                              & = & (I - QQ^*)(AVM_B - BV M_A).
\end{eqnarray}
Thus, $W$ is indeed the projected preconditioned residual of problem~\eqref{eq:gschur2}. 

A common feature of block iterations is that some approximations can converge faster than others. If some columns $v_1, \ldots, v_j$ and $q_1, \ldots, q_j$ in $V$ and $Q$ have 
converged,
they can be ``locked'' using the ``soft locking'' strategy (see, e.g.,~\cite{Kn.Ar.La.Ov:07}), the same way as done for Hermitian eigenproblems. 
With this approach, 
in subsequent iterations, one retains the converged vectors in $V$ and $Q$, but removes the corresponding columns 
from the blocks $W, S_1, \ldots,  S_m$, and $P$.    
Note that locking of the Schur vectors in Algorithm~\ref{alg:gplhr-qz} should always be performed in order,
i.e., $v_{j+1}$ can be locked only after the preceding Schur vectors $v_1, \ldots, v_j$ have converged and been locked.

The use of ``locking'' can have a strong impact on the convergence behavior of the method, especially when a large number of vectors have converged. 
As locking proceeds,
an increasing number of columns of $W, S_1, \ldots,$  $S_m$, and $P$ is removed, 
leading to a shrinkage of the trial subspace.  As a result, the trial subspaces can become excessively
small, 
hindering the
convergence rate and robustness. 
In contrast to the Hermitian eigenproblems, the effect of the subspace reduction can be significantly more
pronounced in the non-Hermitian case, which generally requires searching in subspaces of a larger dimension.

In order to overcome this problem, we propose 
to automatically adjust the parameter $m$. In our implementation, we increase $m$ to the largest integer such that $m (k-q) \leq m_{0} k$, 
where $m_{0}$ is the initial value of $m$ and $q$ is the number of converged eigenpairs. Specifically, we recalculate $m$ as 
$m \gets \min \{ \lfloor m_0 k/(k-q) \rfloor, 20 \}$. 

The GPLHR algorithm can be implemented with $(m+1)$ multiplications of the matrix $A$, and the same number of multiplications of $B$, 
by a block of $k$ vectors per iteration. The allocated memory should be sufficient to store the basis of the trial and test subspaces $Z$ and $U$,
as well as the blocks $AZ$ and $BZ$, whose size depends on the parameter $m$, i.e., 
$s = (m+2)k$ for~\eqref{eq:trial_partial}, 
or $s = (m+3)k$ for the augmented subspace~\eqref{eq:trial_partial1}. 
Thus, to decrease computational cost and storage, the value of $m$ should be chosen as 
small as possible, but, on the other hand, 
should be large enough not to delay
a rapid and stable convergence. For example,
in most of our numerical experiments, we use $m = 1$. 

Note that step 28 of Algorithm~\ref{alg:gplhr-qz} is optional. It provides means to switch between the LOBPCG-style
search directions~\eqref{eq:P_lobpcg} and those defined by~\eqref{eq:P1}. We use the latter by default.  


Finally, we emphasize that Algorithm~\ref{alg:gplhr-qz} is capable of computing larger numbers of Schur vectors incrementally, using \textit{deflation},
similar to Hermitian eigenproblems. 
That is, if $V$ and $Q$ are the computed bases of $k$ right and left generalized Schur vectors, the 
additional $k$ generalized Schur vectors can be computed by applying GPLHR
to the deflated pair $((I - QQ^*) A (I - VV^*), (I - QQ^*) B (I - VV^*))$ if $B \neq I$, or to the matrix
$(I - VV^*) A (I - VV^*)$ otherwise.
This process can be repeated until the desired number of vectors is obtained. 
The ability to facilitate 
the deflation mechanism is an additional advantage of the generalized Schur formulations~\eqref{eq:gschur} and~\eqref{eq:gschur2}, as opposed to 
eigendecomposition~\eqref{eq:partial} for which the deflation procedure is not directly applicable.  

\subsection{Alternative formulas for residual computation}\label{subsec:other_res}
As has been pointed out in Section~\ref{subsec:gschur_new}, the ``Q-free'' generalized Schur form~\eqref{eq:gschur2} is not unique. 
Therefore, the computation of the residual in~\eqref{eq:WA_gschur} can, in principle, be based on a different formula rather than the 
expression $AV^{(i)}M_B^{(i)}-BV^{(i)}M_A^{(i)}$ with $M_A^{(i)},M_B^{(i)}$ given 
by~\eqref{eq:G1_fix},~\eqref{eq:G2_fix} and~\eqref{eq:M}.
For example, if either (or both) of the upper triangular factors $\tilde R_A, \tilde R_B$,
resulting from the harmonic projection, are 
nonsingular at all iterations, which is the most common case in the majority of practical computations,
then the residual can be formulated as $AV^{(i)}(\tilde R_A^{-1} \tilde R_B)-BV^{(i)}$ or $AV^{(i)}-BV^{(i)} (\tilde R_B^{-1} \tilde R_A)$. 

All these formulas are asymptotically equivalent, in the sense that the correspondingly defined residuals vanish as $V^{(i)}$ converges to the desired right Schur vectors. One would naturally question if their choice could affect the performance of GPLHR before it reaches the asymptotic mode. 
Our extensive numerical experience indicates that the effect of the choice of the residual formulas is very mild. Other factors, e.g., the quality of the preconditioner, the subspace dimension parameter $m$, appropriate use of projectors $P_{V^{\perp}}$ and 
$P_{Q^{\perp}}$ for developing the trial subspace, and the choice of the additional search directions $P^{(i)}$ all have much stronger impact on the behavior of GPLHR. Therefore, we always construct the residual as in~\eqref{eq:WA_gschur}, which is well-defined for all regular pencil $(A,B)$. 
The same considerations apply to the $S$-vectors~\eqref{eq:SA_gschur}.

\subsection{The GPLHR algorithm for partial eigenvalue decomposition}\label{subsec:gplhr-eig}
If eigenvalues of $(A, B)$ are non-deficient and the targeted eigenvectors $X$ are known to be 
reasonably well-conditioned, then one would prefer to tackle the 
partial eigenvalue decomposition~\eqref{eq:partial} directly, without resorting to the 
Schur form computations. Such situations, in particular, can arise when the non-Hermitian problem is obtained 
as a result of some perturbation of a Hermitian eigenproblem, e.g., as in the context of 
the EOM-CC~\cite{HeadGordon.Lee:97, el-str-book-Helgaker.Jorgensen.Olsen:2000}
or complex scaling methods~\cite{Balslev.Combes:71, Ho:83, Reinhardt:82} in quantum chemistry. 
Additionally, eigenvectors often carry certain physical information and
it can be of interest to track evolution of their approximations without extra transformations.

It is straightforward to apply the described framework to devise an eigenvector-based algorithm for computing the partial decomposition~\eqref{eq:partial}.
In particular, using exactly the same orthogonal correction argument, 
we can obtain a version of GPLHR that operates on the approximate eigenvectors rather than Schur vectors. This scheme, which we refer to as GPLHR-EIG, is stated in 
Algorithm~\ref{alg:gplhr-eig} of Appendix A.

\section{Preconditioning}\label{sec:prec}
 
The choice of the preconditioner is crucial for the overall performance of the GPLHR methods. While
the discussion of particular preconditioning options is outside the scope of this paper, below, we briefly 
address several general points related to a proper setting of the preconditioning procedure.

An attractive feature of the GPLHR algorithms is that the definition of an appropriate 
preconditioner $T$ is very broad. We only require $T$ to be a form of an approximate inverse of
$(\sigma_B A - \sigma_A B)$. Hence, $T$ can be constructed using any available preconditioning technique 
that makes the solution of the shifted linear systems $(\sigma_B A - \sigma_A B)w=r$
easier to obtain.
For a survey of various preconditioning
strategies we refer the reader 
to~\cite{Benzi.Golub.Liesen:05, Saad:03}. 
Note that, in practice, we expect $T \approx (A - \sigma B)\inv$
if eigenvalues of interest are finite, and $T \approx B^{\dagger}$ if infinite eigenvalues are wanted. For example, in the latter
case, one can let $T \approx  (B + \tau I)\inv$, where $\tau$ is a small regularization parameter.

If $T$ is not sufficient to ensure the eigensolver's convergence with a small search subspace, it can be used as a preconditioner for an approximate iterative 
solve of $(\sigma_B A - \sigma_A B)w=r$, e.g., using the GMRES algorithm~\cite{Saad.Schultz:86} or IDR($s$), a popular short-term recurrence 
Krylov subspace method for nonsymmetric systems~\cite{Sonneveld.Gijzen:08}. To enhance the performance by fully utilizing BLAS-3 operations, 
block linear solvers should be used whenever available.
Then the preconditioned linear solver itself can be used in place of the operator $T$ to precondition the GPLHR algorithm.
Thus, by setting the number of iterations of the linear solver or by adjusting its convergence threshold, one obtains a flexible framework 
for ``refining'' the preconditioning quality, which is especially helpful for 
computing interior eigenvalues.

Finally, if $T$ admits a fast computation, it can be periodically updated in the course of iterations, in such a way that the shift 
in the preconditioning operator is adjusted according to the current eigenvalue approximations. For example, 
this strategy was adopted in the GPLHR implementation~\cite{Zuev.Ve.Yang.Orms.Krylov:15}, where $T$ was diagonal. 


\section{Numerical experiments}\label{sec:num}

The goal of this section is two-fold. First, we would like to demonstrate numerically the significance 
of several algorithmic components of GPLHR, such as the use of 
additional search directions $P^{(i)}$ or the application of a projector prior to preconditioning with $T$.  
Secondly, we are interested in comparing GPLHR with a number of existing state-of-the-art eigensolvers
for non-Hermitian eigenvalue problems, which includes the block GD (BGD), the JDQR/JDQZ algorithm, and the implicitly restarted Arnoldi method
available in ARPACK~\cite{ARPACK:98}. 

\begin{table}[htb]
\centering
{\small
\begin{tabular}{|c|c|c|c|c|c|}
\hline 
Problem & Type & $n$ & $\sigma$ & Preconditioner & Spectr. region \tabularnewline
\hline 
\hline 
A15428 & stand. & $15428$ & $-40.871$ &  ILU($10^{-3}$)  & interior \tabularnewline
\hline 
AF23560 & stand. & $23560$ & $-40 + 300i$ &  ILU($10^{-3}$)  &largest mod. \tabularnewline
\hline 
CRY10000 & stand. & $10000$ & $8.0$ &  ILU($10^{-3}$)  &largest Re\tabularnewline
\hline 
DW8192 & stand. & $8192$ & $1.0$ &  ILU($10^{-3}$)  & rightmost \tabularnewline
\hline 
LSTAB\_NS & gen. & $12619$ & $0$ & GMRES(25)+LSC & rightmost  \tabularnewline
\hline 
MHD4800 & gen. & $4800$ & $-0.1+0.5i$ & $(A - \sigma B)^{-1}$ & interior \tabularnewline
\hline 
PDE2961 & stand. & $2961$ & $10.0$ &  ILU($10^{-3}$)  & largest Re \tabularnewline
\hline 
QH882 & stand. & $882$ & $-150 + 180i$ & GMRES(5)+ILU($10^{-5}$)  & interior \tabularnewline
\hline 
RDB3200L & stand. & $3200$ & $2i$ &  ILU($10^{-3}$)  & rightmost \tabularnewline
\hline 
UTM1700 & gen. & $1700$ & $0$ &  GMRES(15)+ILU($10^{-4}$)  & {leftmost} \tabularnewline
\hline 
\end{tabular}
}
\caption{{\small Test problems.}}
\label{tab:test}
\end{table}

Table~\ref{tab:test} summarizes the test problems considered throughout this section. For all problems, 
the shifts $\sigma$ are chosen to have physical meaning, i.e., they target the part of spectrum that is normally
sought for in practical applications. 
This targeted region of the spectrum is stated in the rightmost column of the table.  

The problems in Table~\ref{tab:test} come from a number of different applications. For~example, the complex 
symmetric matrix A15428 arises from the resonant state calculation
in quantum chemistry~\cite{Balslev.Combes:71, Ho:83, Reinhardt:82}.
Problem LSTAB\_NS is delivered by the linear stability analysis of a model of two dimensional incompressible fluid flow over a backward facing step, constructed using the IFISS software package~\cite{Elman.Ramage.Silvester:07}. The domain is $[-1, L] \times [-1, 1]$ with $[-1, 0]\times[-1, 0]$ cut out, where $L = 10$; the Reynolds number is $300$. 
A biquadratic/bilinear ($Q2$-$Q1$) finite element discretization with element width 
$1/8$ is used.
More details on this example can be found in~\cite{Xue.Elman:12}. 

The rest of the test matrices in Table~\ref{tab:test} are obtained from {\sc Matrix Market}\footnote{\url{http://math.nist.gov/MatrixMarket/}}. 
The three larger problems
AF23560, CRY10000, and DW8192 arise from the stability analysis of the Navier-Stokes equation, crystal growth simulation, and 
integrated circuit design, respectively. MHD4800 is a generalized eigenproblem given by an application in magnetohydrodynamics, 
where interior eigenvalues forming the so-called Alfv\'{e}n branch are wanted. Another interior eigenvalue problem is QH882,
coming from the power systems modeling. The matrices PDE2961 and RDB3200L are obtained from a 2D variable-coefficient linear elliptic partial 
differential equation and a reaction-diffusion Brusselator model in chemical engineering, respectively. UTM1700 corresponds 
to a generalized eigenproblem from nuclear physics (plasmas).   
All of our tests are performed in {\sc Matlab}. 

Throughout, unless otherwise is explicitly stated, the value of the trial subspace size parameter $m$ is set to $1$.
As a preconditioner $T$ we use triangular solves with incomplete $L$ and $U$ factors computed by the {\sc Matlab}'s \texttt{ilu} routine with 
thresholding (denoted by ILU($t$), where $t$ is the threshold value). If a more 
efficient preconditioner is needed then $T$ corresponds to several steps of the 
ILU-preconditioned GMRES, further referred to as GMRES($s$)+ILU($t$), where $s$ is the number of GMRES steps. The only exception is  
LSTAB\_NS, where the action of $T$ is achieved by GMRES(25) preconditioned by an ideal version of the least squares commutator (LSC) preconditioner \cite{Elman.Silvester.Wathen:14}. The default preconditioners for all test problems are also listed in Table~\ref{tab:test}.

\subsection{Effects of the additional search directions}

\begin{table}[htb]
\centering
{\small
\begin{tabular}{|c|c|c|c|}
\hline 
Problem &  $P^{(i)}_{\mbox{{\tiny thick}}}$  & $P^{(i)}_{\mbox{{\tiny LOBPCG}}}$  & W/o $P^{(i)}$  \tabularnewline
\hline 
\hline 
A15428 &  12 & 14 & 16\tabularnewline
\hline 
AF23560 &  10 & DNC & 153\tabularnewline
\hline 
CRY10000 &  27 & DNC & DNC\tabularnewline
\hline 
DW8192 &  118 & 55 & 329 \tabularnewline
\hline 
LSTAB\_NS & 38  & DNC & DNC \tabularnewline
\hline 
MHD4800 &  137 & DNC & DNC \tabularnewline
\hline 
PDE2961 &  13 & DNC & DNC\tabularnewline
\hline 
QH882 &  15 & DNC & 30\tabularnewline
\hline 
RDB3200L &  10 & 12 & 12\tabularnewline
\hline 
UTM1700 & 16  &  23 & 21 \tabularnewline
\hline 
\end{tabular}

}
\caption{{\small Numbers of iterations performed by GPLHR
to compute $k=5$ eigenpairs with different choices of the search directions $P^{(i)}$. 
}}
\label{tab:P}
\end{table}

In Table~\ref{tab:P}, we report numbers of iterations performed by the GPLHR methods to compute five eigenpairs 
with different choices of the search directions~$P^{(i)}$. In particular, we compare the default 
``thick-restarted'' option~\eqref{eq:P1}
with the LOBPCG style choice~\eqref{eq:P_lobpcg}, 
and the variant without~$P^{(i)}$, which corresponds to the non-augmented subspace~\eqref{eq:trial_partial}. 

As has already been mentioned, in practice, the quality of approximations produced by GPLHR should be assessed through 
norms of the Schur residuals. However, in order to facilitate comparisons with other methods, at each step, 
we transform the current approximate Schur basis $V^{(i)}$ into eigenvector approximations and declare convergence
of a given pair $(\lambda, x)$ using its relative eigenresidual norm. That~is, we consider $(\lambda, x)$ 
converged, and soft-lock it from further computations, if $\|A x - \lambda Bx\|/\|Ax\| < 10^{-8}$.  
Note that, in this case, locking is performed in a contiguous fashion,
i.e., eigenvector $x_{j+1}$ can be locked only after the preceding vectors $x_1, \ldots, x_j$ 
have converged and been locked.

Table~\ref{tab:P} clearly demonstrates that the choice $P^{(i)}=V_{k+1:2k}^{(i)}$, 
denoted by $P^{(i)}_{\mbox{{\tiny thick}}}$,
generally leads to the fastest and most 
robust convergence behavior of GPLHR. Such a construction of the search directions requires a negligibly low cost, without
extra matrix-vector products or preconditioning operations. However, their presence in the trial subspace indeed significantly accelerates and 
stabilizes the convergence, as can be observed by comparing the corresponding iteration counts with the case where the block $P^{(i)}$ 
is not appended to the subspace~\eqref{eq:P1}.    
Here, ``DNC'' denotes cases where the method failed to reduce the relative eigenresidual 
norm below the tolerance level of $10^{-8}$ for all eigenpairs within 500 iterations.

Furthermore, it is evident from Table~\ref{tab:P} that the choice 
$P^{(i)}_{\mbox{{\tiny thick}}}$
generally outperforms the LOBPCG style formula~\eqref{eq:P_lobpcg}, denoted by $P^{(i)}_{\mbox{{\tiny LOBPCG}}}$. 
Specifically, the former always leads to convergence and yields lower iteration counts for all of the test problems,
except for DW8192. In our experience, however, this situation is not common, and the use of~\eqref{eq:P1} 
typically results in a much more robust convergence behavior. 

\subsection{Expansion of the trial subspace}

The size of the GPLHR trial subspace is controlled by the parameter $m$. Generally, increasing $m$ 
leads to a smaller iteration count. However, since the 
expansion 
of the GPLHR trial subspace 
requires 
extra matrix-vector multiplications and 
preconditioning 
operations, using larger values of $m$ also increases the cost of each GPLHR step.
Therefore, the parameter $m$ should be chosen large enough to ensure stable convergence, 
but at the same time should be sufficiently small to maintain a relatively low cost of iterations.

\begin{figure}[ht]
\begin{center}%
    \includegraphics[width=6.4cm]{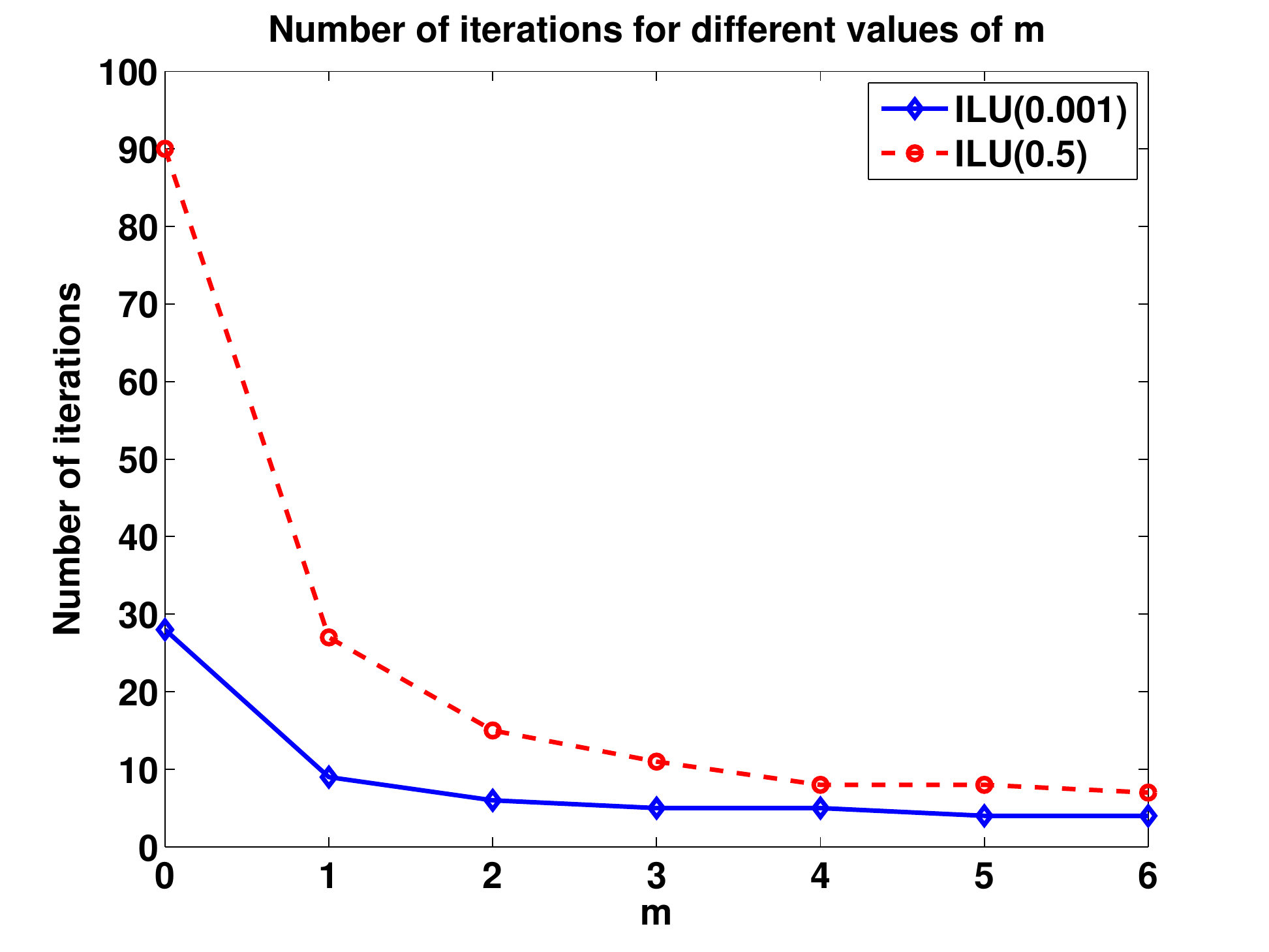}
    \includegraphics[width=6.4cm]{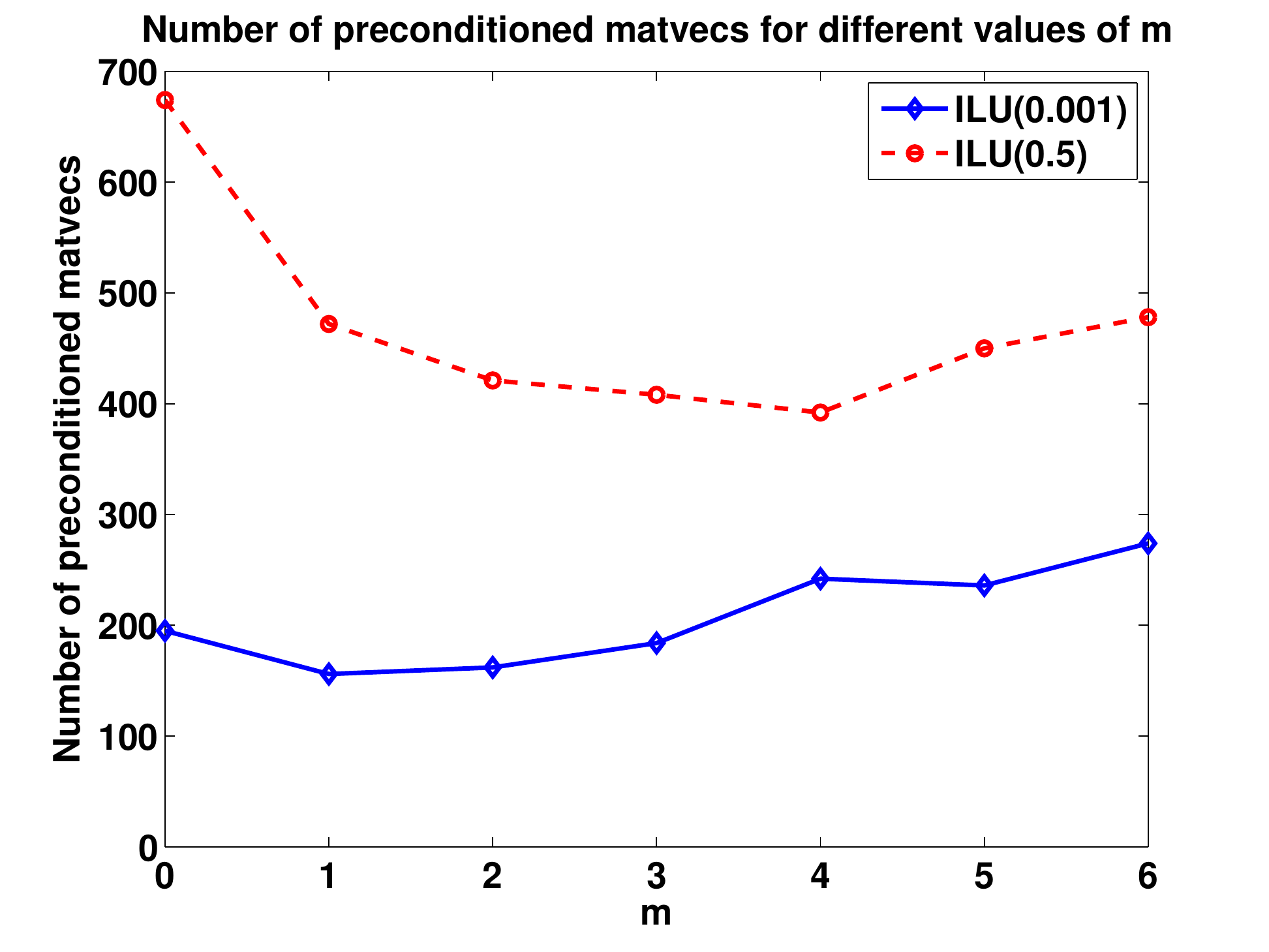}
\end{center}
\caption{
{\small 
Numbers of iterations (left) and preconditioned matrix-vector products (right) performed by GPLHR,
preconditioned with ILU\textup{($0.001$)} and ILU\textup{($0.5$)},
to compute $k=10$ eigenpairs of AF23560 with different values of the subspace parameter $m$. 
}
}
\label{fig:m}
\end{figure}

In most cases, using larger values of $m$ is counterproductive, as can be seen from Figure~\ref{fig:m}. The figure shows 
how the number of iterations and preconditioned matrix-vector products change as $m$ grows. If a good preconditioner is at hand,
such as ILU($0.001$) for AF23560 in the figure, then
it is common that the iteration count stabilizes at $m=1$ or $m=2$ and therefore further increase of the parameter is not needed
as it leads to redundant computations. However, increasing~$m$ can be more effective when 
less efficient preconditioners are used,
e.g., ILU($0.5$) in our test. In this case, 
a relatively low preconditioning quality is compensated by a larger trial subspace.
In particular, in the AF23560 example with ILU($0.5$),
the optimal value of $m$ in terms of computational work is $4$. Here, again, the convergence tolerance, with respect to the
relative eigenresidual norms, is set to $10^{-8}$.   

\subsection{Effects of projectors}
Our derivation of the GPLHR method in Section~\ref{sec:pqr}
suggested that 
the 
preconditioner~$T$ should be preceded by application of the projector $(I - Q^{(i)}Q^{(i)*})$ 
(or $(I - V^{(i)}V^{(i)*})$, if $B = I$), and then followed by the projection $(I - V^{(i)}V^{(i)*})$
(see~\eqref{eq:prec_gen} and~\eqref{eq:prec}). 
This motivated the occurrence of the corresponding projectors on both sides of $T$ at 
steps 8 and 15 (if $B \neq I$) and at steps 10 and 17 (if $B = I$) of Algorithm~\ref{alg:gplhr-qz}. 

While the presence of the left projector $I - V^{(i)}V^{(i)*}$ can be viewed as a part of the orthogonalization procedure on the trial subspace,
and is thus natural, 
it is of interest to see whether the action of the right projection, performed before the application of $T$, has any  
effect on the GPLHR convergence. For this reason, we compare the GPLHR scheme in Algorithm~\ref{alg:gplhr-qz} 
to its variant where the right projector is omitted. The corresponding iteration counts are reported in Table~\ref{tab:proj}.


\begin{table}[htb]
\centering
{\small
\begin{tabular}{|c|c|c|}
\hline 
Problem & With proj.  & W/o proj.\tabularnewline
\hline 
\hline 
A15428 &  14 & DNC \tabularnewline
\hline 
AF23560 &  9 & 9 \tabularnewline
\hline 
CRY10000 &  13 & 17\tabularnewline
\hline 
DW8192 &  44 & 45\tabularnewline
\hline 
LSTAB\_NS & 49  & 84 \tabularnewline
\hline 
MHD4800 &  12 & 12\tabularnewline
\hline 
PDE2961 & 10 & 11 \tabularnewline
\hline 
QH882 &  47 & 43 \tabularnewline
\hline 
RDB3200L &  12 & 21 \tabularnewline
\hline 
UTM1700 & 25  & 26 \tabularnewline
\hline 
\end{tabular}

}
\caption{{\small Numbers of iterations performed by GPLHR with and without
the projection at the preconditioning step to compute $k=10$ eigenpairs.
}}
\label{tab:proj}
\end{table}

It can be observed from Table~\ref{tab:proj} that, in most of the tests, 
applying the projector before applying
applying $T$ gives a smaller number of iterations. In several cases, the difference in the iteration count is substantial;
see the results for A15428, LS\_STAB, and RDB3200L. 
The only example, where the absence of the projector yields slightly less iterations is QH882. However, this only happens for 
a small value of the parameter $m$, and the effects of the projection become apparent and favorable as $m$ is 
increased and the convergence is stabilized.

%

\subsection{Comparison with BGD}
We now compare GPLHR to the classic BGD method used to solve non-Hermitian problems~\cite{Morgan:92, Sadkane:93}. 
This eigensolver is popular, in particular, in quantum chemistry; e.g.,~\cite{Zuev.Ve.Yang.Orms.Krylov:15}. 
For a fair comparison, in BGD,
we apply the harmonic Rayleigh--Ritz procedure~\cite{Morgan.Zeng:98} 
to extract approximate eigenvectors and use the same projected 
preconditioner~\eqref{eq:prec_gen} or~\eqref{eq:prec}, depending on whether the eigenproblem is generalized or standard, respectively. In fact, once equipped with the projected preconditioning, BGD can be viewed as a generalization of the 
JD methods~\cite{Fokkema:Sleijpen.Vorst:98, Sleijpen.Vorst:96} to the block case. 

To ensure that both schemes have the same memory constraint, we consider a restarted variant of BGD, where the 
search subspace is collapsed after reaching the dimension of $(m+3)k$, with the $k$ current approximate eigenvectors 
(the harmonic Ritz vectors) used to start the next cycle. We denote this scheme as BGD($(m+3)k$). 
In particular, if $m=1$, we get BGD($4k$).
 

In principle, it is hard to compare preconditioned eigensolvers without knowledge of a specific preconditioner used within the given 
application. Therefore, our first experiment, reported in Figure~\ref{fig:quant_chem}, attempts to assess the methods under different preconditioning 
quality, where $T$ ranges from a fairly crude approximation of $(A - \sigma B)^{-1}$ to a 
more accurate approximation.
\begin{figure}[ht]
\begin{center}%
    \includegraphics[width=6.4cm]{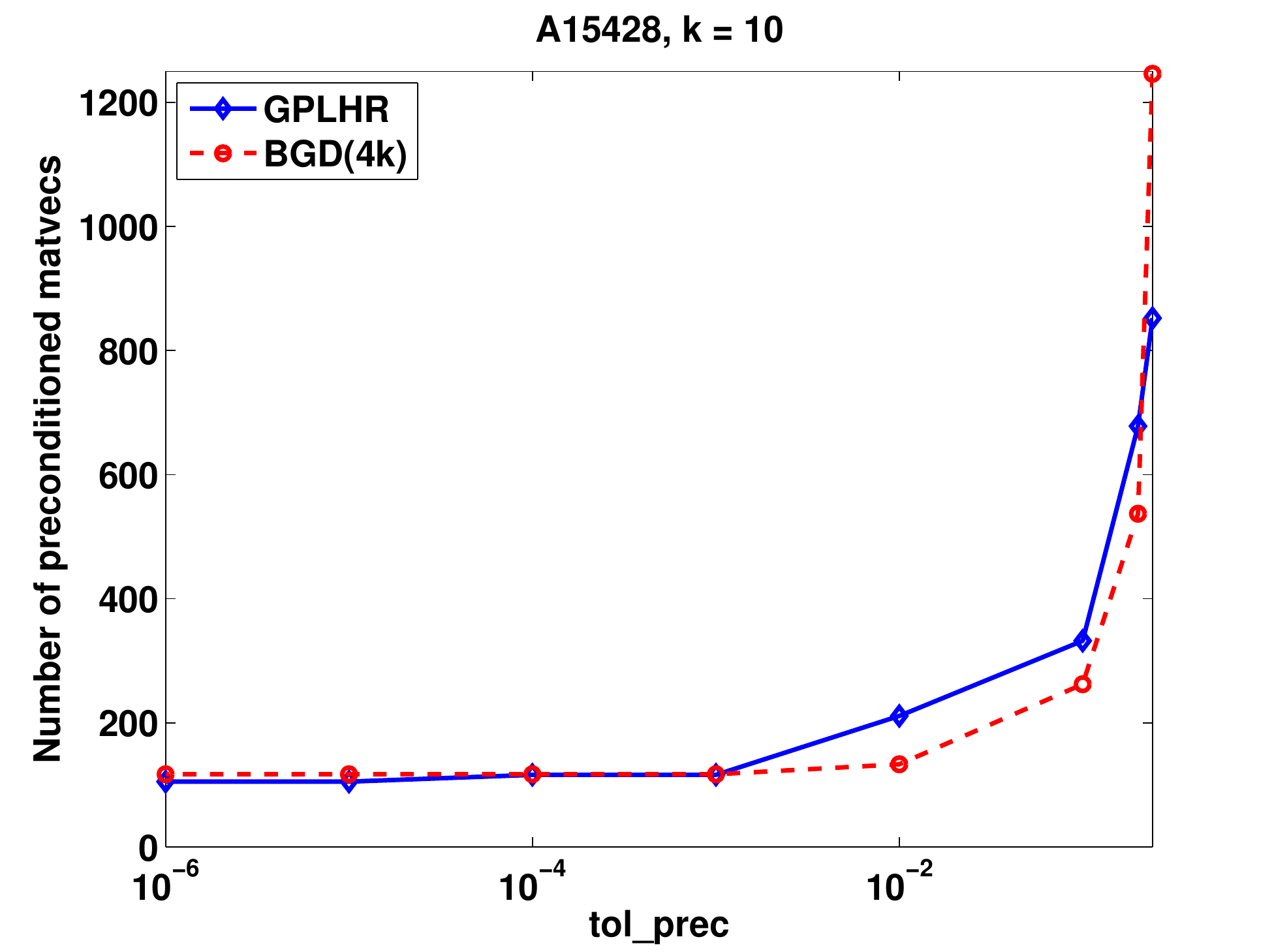}
    \includegraphics[width=6.4cm]{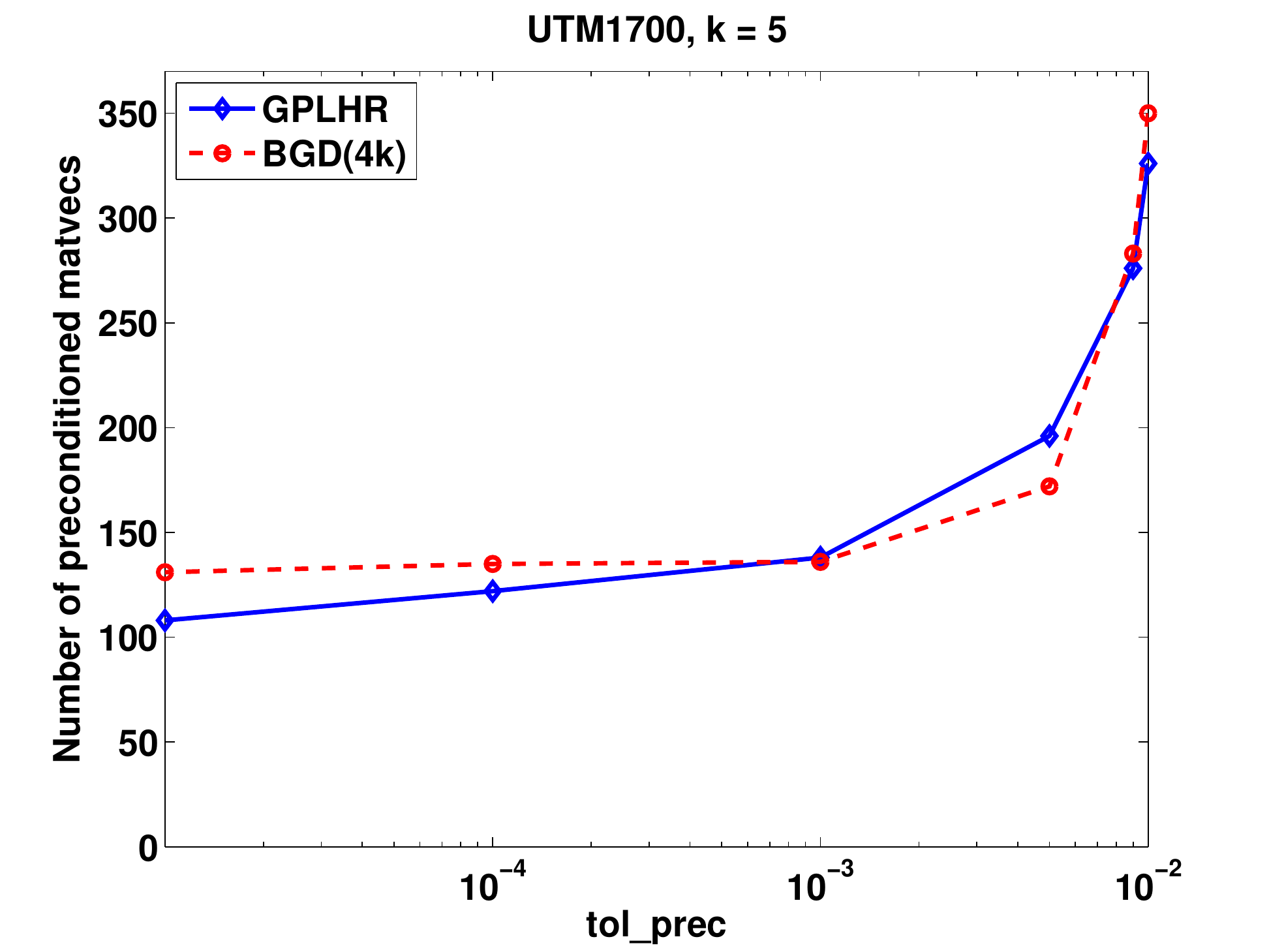}
\end{center}
\caption{
{\small
Comparison of GPLHR \textup{($m=1$)} and BGD\textup{($4k$)} with different preconditioning quality for computing several eigenpairs 
of A15428 (left) and UTM1700 (right) closest to~$\sigma$. The preconditioner
$T$ corresponds to an application of the ILU-preconditioned GMRES with stopping tolerance \textit{tol\_prec}.
}
}
\label{fig:quant_chem}
\end{figure}

In order to emulate a variety of preconditioners, we set $T$ to an approximate solve of $(A - \sigma B)w = r$ using the (full) 
preconditioned GMRES. The quality of $T$ can then be adjusted by varying the GMRES convergence tolerance, denoted by \textit{tol\_prec}.
Clearly, lower values \textit{tol\_prec} correspond to a 
more effective preconditioner,
whereas increasing \textit{tol\_prec} leads to 
its deterioration.

Figure~\ref{fig:quant_chem} shows the numbers of (preconditioned) matrix-vector products 
for different levels of preconditioning quality.
Here, GPLHR and BGD($4k$) are applied for computing several eigenpairs of the problems A15428 (left) and UTM1700 (right).
The reported preconditioned matrix-vector product (matvec) counts include only those performed by the eigensolvers, 
without the ones from the ``inner'' GMRES iterations, which are hidden in the application of $T$. 
As a GMRES preconditioner we use ILU($10^{-2}$) for A15428 and ILU($10^{-4}$) for UTM1700.   

The results in Figure~\ref{fig:quant_chem} are representative. 
They demonstrate that, under the same memory constraint, 
GPLHR tends to exhibit a better performance if the preconditioner is either reasonably strong 
(i.e., \textit{tol\_prec}$\leq 10^{-3}$ in both plots) or is relatively weak (i.e.,
$\textit{tol\_prec}$ around $10^{-1}$ on the left plot and around $10^{-2}$
on the right). 
At~the same time, one can see that BGD($4k$) can outperform GPLHR for an 
``intermediate''
preconditioning~quality, i.e., corresponding to \textit{tol\_prec} between $10^{-3}$ and $10^{-1}$ in the left,
and between $10^{-3}$ and $0.009$ on the right, plot of the~figure.

\begin{table}[htb]
\centering
{\small
\begin{tabular}{|l|c||c|c||c|c|}
\cline{3-6} 
\multicolumn{1}{c}{} & \multicolumn{1}{c|}{} & \multicolumn{2}{c||}{GPLHR} & \multicolumn{2}{c|}{BGD($4k$)}\tabularnewline
\hline 
Problem & $k$ & \#it  & \#pmv & \#it & \#pmv\tabularnewline
\hline 
\hline 
A15428 & 3 & 17 & 98 & DNC & DNC\tabularnewline
\cline{2-6} 
 & 5 & 11 & 100 & DNC & DNC\tabularnewline
\cline{2-6} 
 & 10 & 14 & 239 & 41 & 220\tabularnewline
\hline 
\hline 
AF23560 & 3 & 10 & 52 & DNC & DNC\tabularnewline
\cline{2-6} 
 & 5 & 10 & 84 & 48 & 118\tabularnewline
\cline{2-6} 
 & 10 & 9 & 156 & 39 & 232\tabularnewline
\hline 
\hline 
CRY10000 & 3 & 29 & 146 & 142 & 237\tabularnewline
\cline{2-6} 
 & 5 & 24 & 192 & DNC & DNC\tabularnewline
\cline{2-6} 
 & 10 & 14 & 245 & DNC & DNC\tabularnewline
\hline 
\hline 
DW8192 & 3 & 70 & 377 & DNC & DNC\tabularnewline
\cline{2-6} 
 & 5 & 51 & 474 & 363 & 1447\tabularnewline
\cline{2-6} 
 & 10 & 48 & 799 & 450 & 1808\tabularnewline
\hline 
\hline 
LSTAB\_NS & 3 & 51 & 5616 & 306 & 8843\tabularnewline
\cline{2-6} 
 & 5 & 37 & 6552 & DNC & DNC\tabularnewline
\cline{2-6} 
 & 10 & 41 & 15184 & DNC & DNC\tabularnewline
\hline 
\hline 
MHD4800 & 3 & 30 & 158 & DNC & DNC\tabularnewline
\cline{2-6} 
 & 5 & 170 & 1062 & DNC & DNC\tabularnewline
\cline{2-6} 
 & 10 & 12 & 206 & 79 & 432\tabularnewline
\hline 
\hline 
PDE2961 & 3 & 12 & 70 & 53 & 123\tabularnewline
\cline{2-6} 
 & 5 & 11 & 98 & 46 & 114\tabularnewline
\cline{2-6} 
 & 10 & 11 & 178 & 50 & 207\tabularnewline
\hline 
\hline 
QH882 & 3 & 18 & 588 & 23 & 348\tabularnewline
\cline{2-6} 
 & 5 & 14 & 708 & 22 & 450\tabularnewline
\cline{2-6} 
 & 10 & 52 & 4332 & 92 & 2190\tabularnewline
\hline 
\hline 
RDB3200L & 3 & 8 & 46 & 16 & 41\tabularnewline
\cline{2-6} 
 & 5 & 9 & 80 & 21 & 76\tabularnewline
\cline{2-6} 
 & 10 & 11 & 180 & 33 & 163\tabularnewline
\hline 
\hline 
UTM1700 & 3 & 17 & 1312 & 68 & 1760\tabularnewline
\cline{2-6} 
 & 5 & 16 & 2080 & 44 & 2208\tabularnewline
\cline{2-6} 
 & 10 & 25 & 6048 & 125 & 6704\tabularnewline
\hline 
\end{tabular}
}
\caption{{\small The numbers of iterations and preconditioned matrix-vector products 
performed by GPLHR \textup{($m=1$)} and BGD\textup{($4k$)}. 
}}
\label{tab:GD}
\end{table}

In Table~\ref{tab:GD}, we compare the numbers of iterations (\#it) and 
(preconditioned) matvecs (\#pmv) performed by GPLHR and BGD($4k$)
for all of our test problems with the default preconditioners listed in Table~\ref{tab:test}.
Here, the matvec counts also include those accrued at ``inner'' GMRES iterations 
within the preconditioning step.  

It can be seen from Table~\ref{tab:GD} that GPLHR is generally more robust than~BGD($4k$).
The former was always able to attain the requested level of the solution accuracy, whereas the latter 
sometimes failed to converge. Furthermore, one can see that in the examples where 
BGD($4k$) is successful, GPLHR typically performs fewer preconditioned matvecs and, hence, is more
efficient; with the only exceptions being 
the QH882 matrix and, to a lesser extent, RDB3200L. 

In contrast to GPLHR, 
BGD can perform the approximate eigenvector extraction from a trial subspace 
that expands throughout the iterations,
while preserving the same number of preconditioned matrix-vector products per step. 
Therefore, the BGD convergence can be accelerated if additional storage is available. 
However, the memory usage increase, sufficient to reproduce the GPLHR performance, can often be 
significant in practice and thus prohibitive for problems of a very large size.

\begin{figure}[ht]
\begin{center}%
    \includegraphics[width=6.4cm]{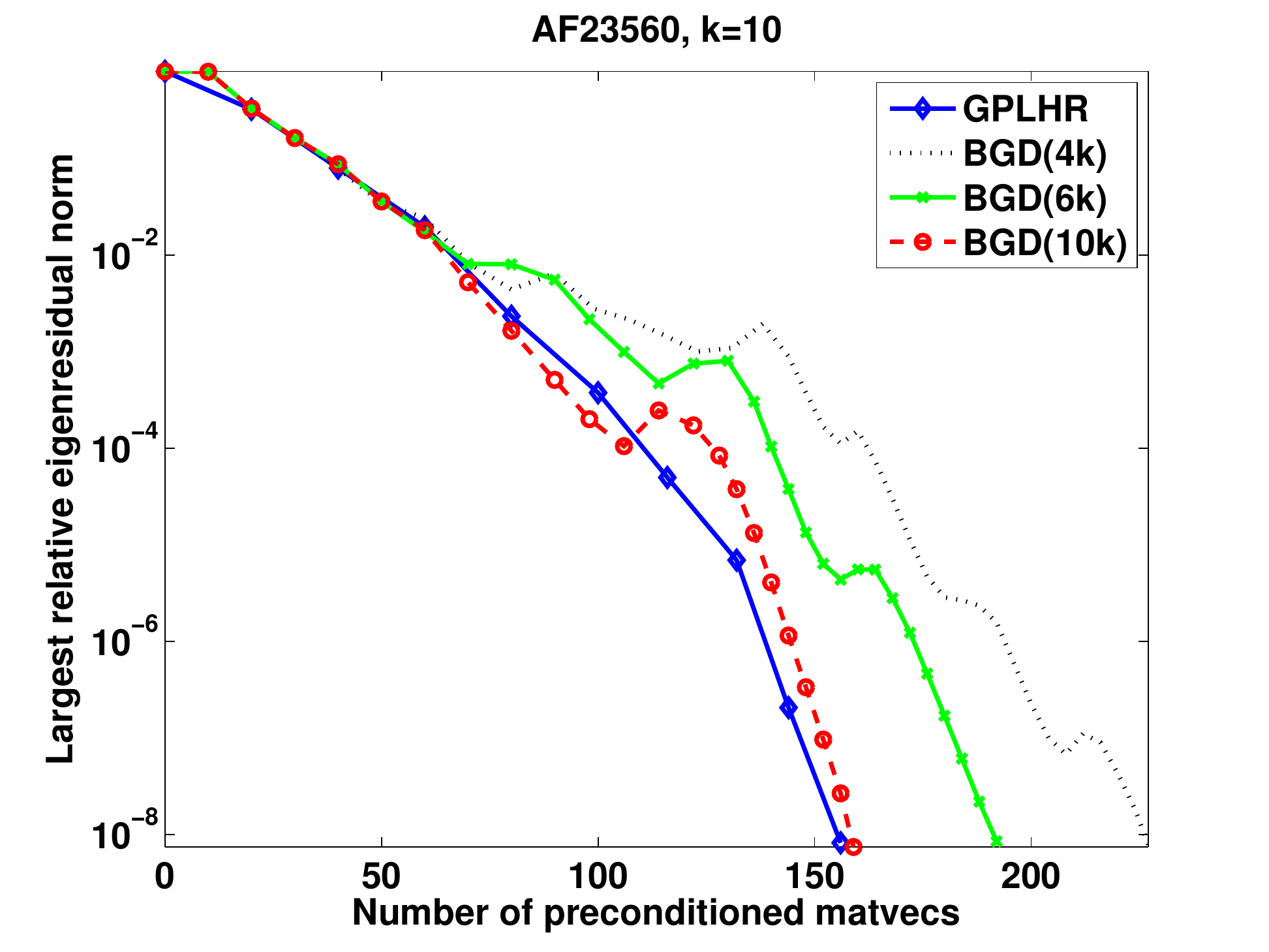}
    \includegraphics[width=6.4cm]{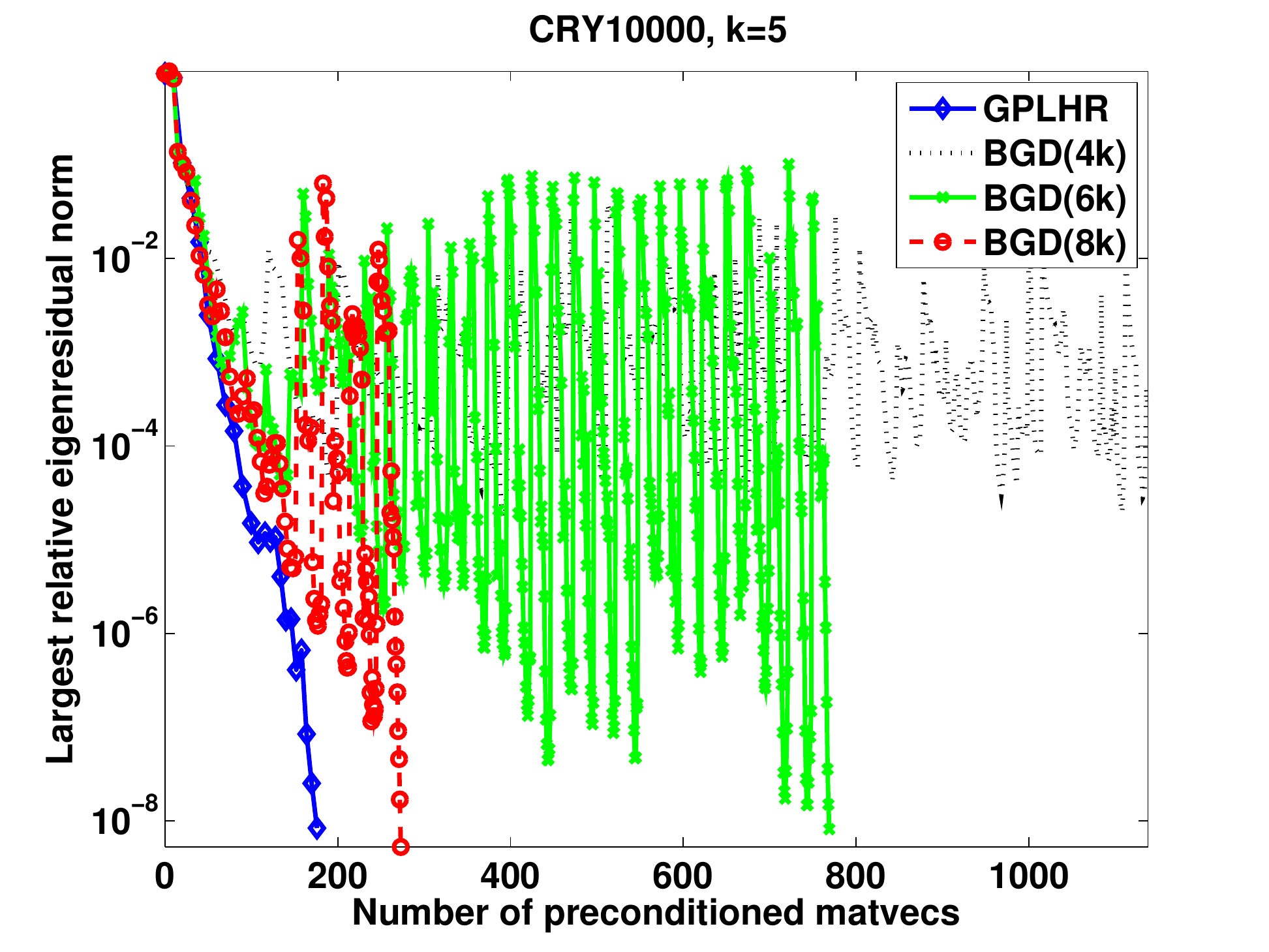} 
    \includegraphics[width=6.4cm]{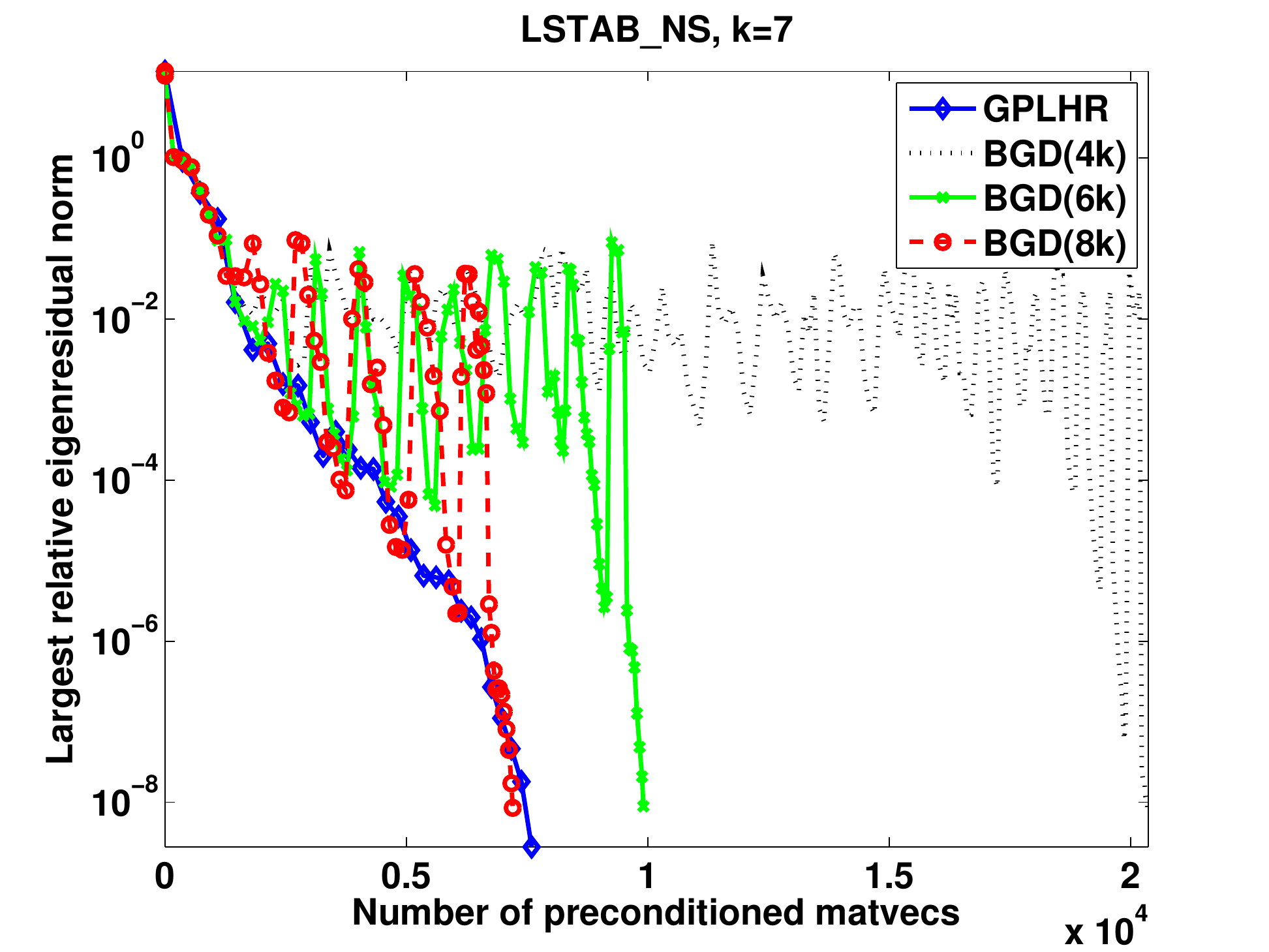}
    \includegraphics[width=6.4cm]{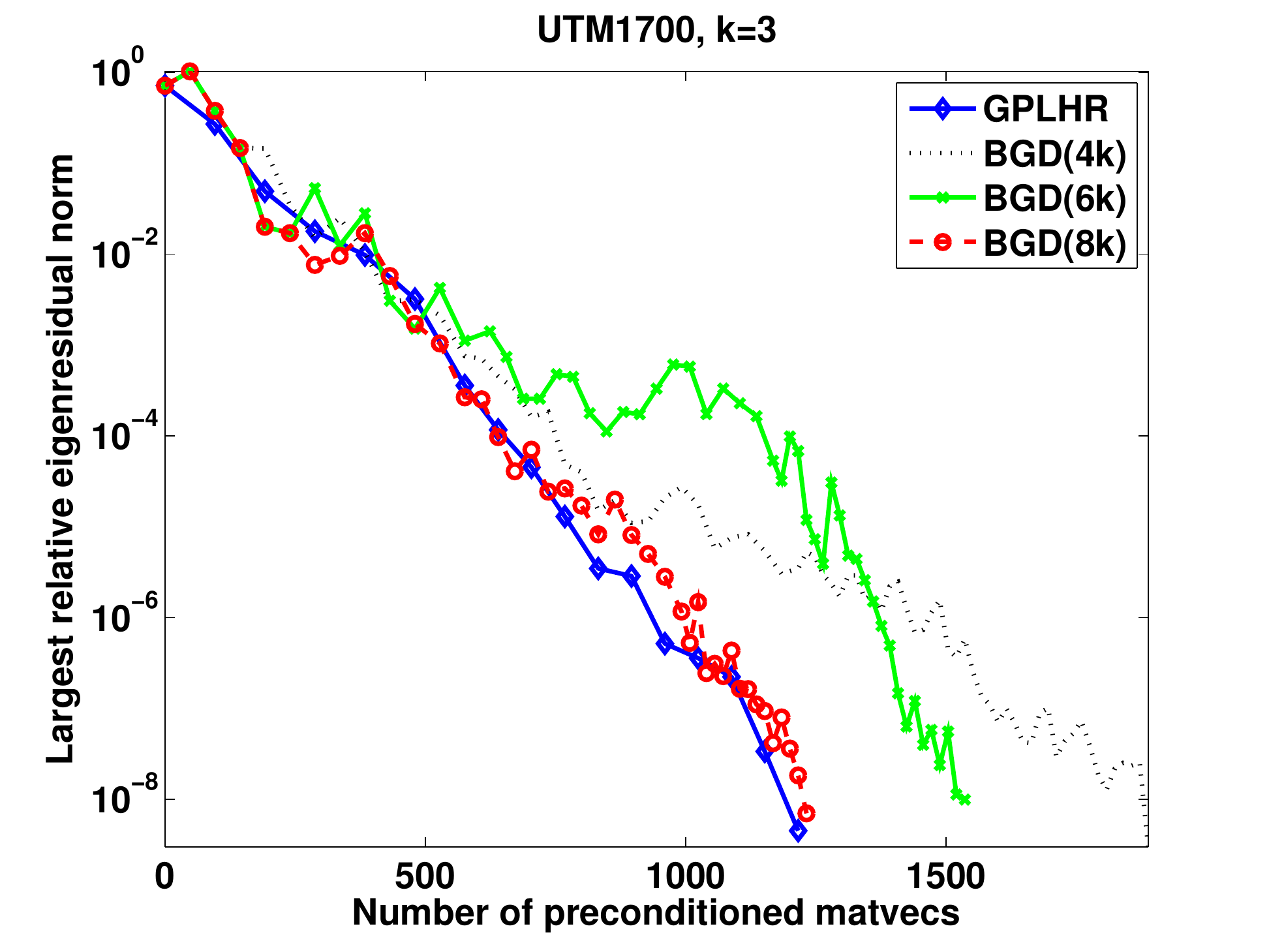}
\end{center}
\caption{
{\small
Comparison of GPLHR \textup{($m=1$)} and BGD\textup{($t$)} with different values of the restart parameter $t$.
}
}
\label{fig:CRY10000_AF23560}
\end{figure}

This point is demonstrated in Figure~\ref{fig:CRY10000_AF23560}, which shows that BGD can require 
around two to three times more storage than GPLHR
to achieve a comparable convergence rate. Thus, the GPLHR schemes can be preferred to BGD in cases where memory is tight.    
For all the test problems in this example we use the default preconditioners listed in Table~\ref{tab:test}.




\subsection{Comparison with JDQR/JDQZ}
As has been mentioned in introduction, the GPLHR method is substantially different from 
JDQR/JDQZ.
The former is based on block iterations, whereas the latter are single vector schemes 
which compute one eigenpair after another in a sequential fashion. 
The goal of our experiments below is to point out the advantages of using the block GPLHR iterations
for practical large-scale eigenvalue computations. 
In our comparisons, we use the \texttt{jdqr.m} 
and \texttt{jdqz.m} {\sc matlab} implementations of the JDQR/JDQZ 
algorithms\footnote{\url{http://www.staff.science.uu.nl/~sleij101/}}.
  
As a GPLHR preconditioner $T$ we employ a fixed number of iterations of preconditioned GMRES (\textit{$it_{G}$})
for the system $(A - \sigma B)w = r$, where $r$ is a vector to which the preconditioner $T$ is applied. 
Throughout, we use the standard single-vector GMRES to solve systems corresponding to different right-hand sides
separately. However, in practice, it may be more efficient to use 
block linear solvers for a simultaneous solution of all~systems. 
In the JDQR/JDQZ context, the preconditioner is applied through solving 
the correction equation. Therefore, in order to match the complexity and~quality of the GPLHR preconditioning to the correction equation solve, the 
solution of the latter is approximated by applying \textit{$it_{G}$} steps of GMRES with the same preconditioner. 
The number \textit{$it_{G}$} is set to the smallest value that yields the convergence of GPLHR.

In our tests, we use the harmonic Rayleigh--Ritz option for eigenvector extraction in JDQR/JDQZ.  
To ensure the same memory constraint for both methods, we set the maximum size of the JDQR/JDQZ search subspace to $4k$ (the parameter $m$ in GPLHR is set to~$1$ 
in all of the runs). After restarting, the JDQR/JDQZ retains $k+5$ harmonic Ritz vectors, which is a default value in the \texttt{jdqr.m} and \texttt{jdqz.m}.

In order to have a similar locking of the converged Schur vectors, we modify the GPLHR stopping criterion to be 
the same as in \texttt{jdqr.m} and \texttt{jdqz.m}. Namely, we~consider the right Schur vector $v_j$ and an approximate eigenvalue $\lambda_j$ 
to be converged only if the preceding vectors
$\bar V = [v_1, v_2, \ldots, v_{j-1}]$ have converged and the norm of the eigenresidual 
of the deflated problem $((I - \bar Q \bar Q^*)A(I - \bar V \bar V^*),(I - \bar Q \bar Q^*)B(I - \bar V \bar V^*))$, evaluated at 
$(\lambda_j, v_j)$, is below the given tolerance level, set to $10^{-8}$ in our tests. 
Here, the matrix $\bar Q = [q_1, q_2, \ldots, q_{j-1}]$ contains the converged 
left Schur vectors, and $\bar Q = \bar V$, if $B = I$.

\begin{table}[htb]

\centering
{\small

\begin{tabular}{|c|c|c||c|c|c||c|c|c|}
\cline{4-9} 
\multicolumn{1}{c}{} & \multicolumn{1}{c}{} & \multicolumn{1}{c|}{} & \multicolumn{3}{c||}{GPLHR} & \multicolumn{3}{c|}{JDQR}\tabularnewline
\hline 
Problem & $k$ & $it_{G}$ & \#it  & \#mv & \#prec & \#it & \#mv & \#prec\tabularnewline
\hline 
\hline 
A15428 & 1 & 15 & 13 & 417 & 416 & 20 & 321 & 341\tabularnewline
\cline{2-9} 
 & 3 & 10 & 16 & 949 & 946 & 45 & 502 & 541\tabularnewline
\cline{2-9} 
 & 5 & 10 & 14 & 1347 & 1342 & 57 & 642 & 685\tabularnewline
\hline 
\hline 
AF23560 & 1 & 0 & 13 & 27 & 26 & DNC & DNC & DNC\tabularnewline
\cline{2-9} 
 & 3 & 0 & 14 & 75 & 72 & 50 & 57 & 101\tabularnewline
\cline{2-9} 
 & 5 & 0 & 14 & 117 & 112 & 69 & 84 & 139\tabularnewline
\cline{2-9} 
 & 10 & 0 & 12 & 214 & 204 & 112 & 147 & 225\tabularnewline
\hline 
\hline 
CRY10000 & 1 & 3 & 46 & 369 & 368 & 62 & 249 & 311\tabularnewline
\cline{2-9} 
 & 3 & 3 & 58 & 1243 & 1240 & 182 & 735 & 911\tabularnewline
\cline{2-9} 
 & 5 & 4 & 45 & 1915 & 1910 & 218 & 1105 & 1309\tabularnewline
\cline{2-9} 
 & 10 & 3 & 42 & 3102 & 3092 & DNC & DNC & DNC\tabularnewline
\hline 
\hline 
DW8192 & 1 & 0 & 175 & 351 & 350 & 659 & 660 & 1318\tabularnewline
\cline{2-9} 
 & 3 & 0 & 81 & 423 & 420 & 274 & 281 & 549\tabularnewline
\cline{2-9} 
 & 5 & 0 & 59 & 535 & 530 & 270 & 258 & 541\tabularnewline
\cline{2-9} 
 & 10 & 0 & 49 & 831 & 821 & 336 & 371 & 673\tabularnewline
\hline 
\hline 
PDE2961 & 1 & 7 & 48 & 769 & 768 & DNC & DNC & DNC\tabularnewline
\cline{2-9} 
 & 3 & 0 & 20 & 117 & 114 & 48 & 55 & 97\tabularnewline
\cline{2-9} 
 & 5 & 0 & 17 & 149 & 144 & 57 & 72 & 115\tabularnewline
\cline{2-9} 
 & 10 & 0 & 15 & 242 & 232 & 86 & 121 & 173\tabularnewline
\hline 
\hline 
QH882 & 1 & 5 & 24 & 289 & 288 & 26 & 157 & 183\tabularnewline
\cline{2-9} 
 & 3 & 5 & 24 & 783 & 780 & DNC & DNC & DNC\tabularnewline
\cline{2-9} 
 & 5 & 5 & 17 & 893 & 888 & 41 & 261 & 288\tabularnewline
\cline{2-9} 
 & 10 & 5 & 50 & 4294 & 4284 & DNC & DNC & DNC\tabularnewline
\hline 
\hline 
RDB3200L & 1 & 0 & 31 & 63 & 62 & 54 & 55 & 109\tabularnewline
\cline{2-9} 
 & 3 & 0 & 20 & 121 & 118 & 71 & 78 & 143\tabularnewline
\cline{2-9} 
 & 5 & 0 & 22 & 213 & 208 & 102 & 117 & 205\tabularnewline
\cline{2-9} 
 & 10 & 0 & 20 & 394 & 384 & 183 & 218 & 367\tabularnewline
\hline 
\end{tabular}

}
\caption{
{\small
Comparison of GPLHR and JDQR algorithms for standard eigenproblems.
}}
\label{tab:jdqr}
\end{table}

In Table~\ref{tab:jdqr}, we report results for standard eigenproblems, where GPLHR is compared with the JDQR algorithm. 
The table contains numbers of eigensolvers' iterations ($\#$it), matrix-vector products ($\#$mv), and applications 
of the preconditioner ($\#$prec) used within GMRES solves.   
In all runs, we apply \textit{$it_{G}$} GMRES steps with 
ILU($10^{-2}$) as the GPLHR preconditioner~$T$ and as the JD correction equation solve, except for the QH882
problem where a stronger ILU($10^{-5}$) preconditioner is used within GMRES. 
Note that 
the zero number \textit{$it_{G}$} means that $T$ is applied through a single shot of the ILU preconditioner,
without the ``inner'' GMRES steps.


One can see from Table~\ref{tab:jdqr} that GPLHR generally performs significantly less iterations than 
JDQR. At the same time, each GPLHR iteration is more expensive, requiring $(m+1)k$ matrix-vector products and preconditioning operations,
while JDQR only performs a single matrix-vector multiply plus the correction equation solve per step. As a result, the total number
of matrix-vector products and preconditionings is, in most cases, larger for GPLHR, especially if larger numbers $k$ of 
eigenpairs are wanted. However, generally, the increase is mild, typically around a factor of 2. On the other hand, GPLHR allows for a possibility 
to group matrix-vector products into a single matrix-block multiply in actual parallel implementation, which can lead to a 2 to 3
times speedup, e.g.,~\cite{Aktulga.Buluc.Williams.Yang:14}, compared to separate multiplications involving single vectors, as in JDQR. 
Therefore, we
expect that GPLHR schemes will ultimately outperform the JDQR/JDQZ family 
for large-scale eigenproblems 
on modern parallel computers. Optimal parallel implementations of GPLHR are, however, outside the scope of this paper 
and will be subject of further study.  

Note that, if \textit{$it_{G} = 0$}, JDQR often produces a larger $\#$prec count, i.e., it requires more   
ILU applications compared to GPLHR.
This is because of the way the correction equation is solved in JDQR/JDQZ~\cite[Section 2.6]{Fokkema:Sleijpen.Vorst:98},
which requires an extra application of the preconditioner to form the correction equation. Thus, one can expect that 
GPLHR gives a faster time to solution in cases where matrix-vector products are relatively cheap, but preconditioning is expensive.   

\begin{table}[htb]

\centering
{\small
\begin{tabular}{|c|c|c||c|c|c||c|c|c|}
\cline{4-9} 
\multicolumn{1}{c}{} & \multicolumn{1}{c}{} & \multicolumn{1}{c|}{} & \multicolumn{3}{c||}{GPLHR} & \multicolumn{3}{c|}{JDQZ}\tabularnewline
\hline 
Problem & $k$ & $it_{G}$ & \#it  & \#mv & \#prec & \#it & \#mv & \#prec\tabularnewline
\hline 
\hline 
MHD4800 & 1 & 0 & 20 & 41 & 40 & DNC & DNC & DNC\tabularnewline
\cline{2-9} 
 & 3 & 0 & 15 & 79 & 76 & DNC & DNC & DNC\tabularnewline
\cline{2-9} 
 & 5 & 0 & 12 & 101 & 96 & DNC & DNC & DNC\tabularnewline
\cline{2-9} 
 & 10 & 0 & 9 & 165 & 155 & DNC & DNC & DNC\tabularnewline
\hline 
\hline 
LSTAB\_NS & 1 & 25 & 5 & 261 & 260 & DNC & DNC & DNC\tabularnewline
\cline{2-9} 
 & 3 & 25 & 49 & 5307 & 5304 & DNC & DNC & DNC\tabularnewline
\cline{2-9} 
 & 5 & 25 & 28 & 4893 & 4888 & DNC & DNC & DNC\tabularnewline
\cline{2-9} 
 & 10 & 25 & 36 & 12984 & 12974 & DNC & DNC & DNC\tabularnewline
\hline 
\hline 
UTM1700 & 1 & 15 & 5 & 161 & 160 & MCV & MCV & MCV\tabularnewline
\cline{2-9} 
 & 3 & 15 & 13 & 963 & 960 & 16 & 263 & 280\tabularnewline
\cline{2-9} 
 & 5 & 15 & 13 & 1669 & 1664 & 31 & 511 & 545\tabularnewline
\cline{2-9} 
 & 10 & 15 & 21 & 5034 & 5024 & 166 & 2691 & 2865\tabularnewline
\hline 
\end{tabular}

}
\caption{
{\small
Comparison of GPLHR and JDQZ algorithms for generalized eigenproblems.
}}
\label{tab:jdqz}
\end{table}

Table~\ref{tab:jdqz} contains the comparison results for generalized eigenproblems, where GPLHR is compared with the 
JDQZ algorithm.
It reports the same quantities as in Table~\ref{tab:jdqr}. Similarly, $\#$prec corresponds 
to the number of preconditioner applications within the \textit{$it_{G}$} GMRES iterations used as a GPLHR
preconditioner $T$. In MHD4800, we set the GMRES preconditioner to $(A - \sigma B)\inv$, in LSTAB\_NS to the ideal least squares commutator preconditioner, and to ILU($10^{-4}$) in UTM1700.       

Table~\ref{tab:jdqz} shows that, for a given preconditioner and trial subspace size, JDQZ failed to converge for MHD4800
and LSTAB\_NS. Slightly increasing \textit{$it_{G}$} allowed us to restore the JDQZ convergence for MHD4800. However, even with a
larger number of \textit{$it_{G}$}, the method misconverged, i.e., missed several wanted eigenpairs and instead returned
the ones that are not among the $k$ closest to $\sigma$. For LSTAB\_NS, we were not able to restore the convergence 
neither by increasing the number \textit{$it_{G}$} of GMRES iterations, nor by allowing a larger trial subspace. 
\begin{figure}[ht]
\begin{center}%
    \includegraphics[width=6.4cm]{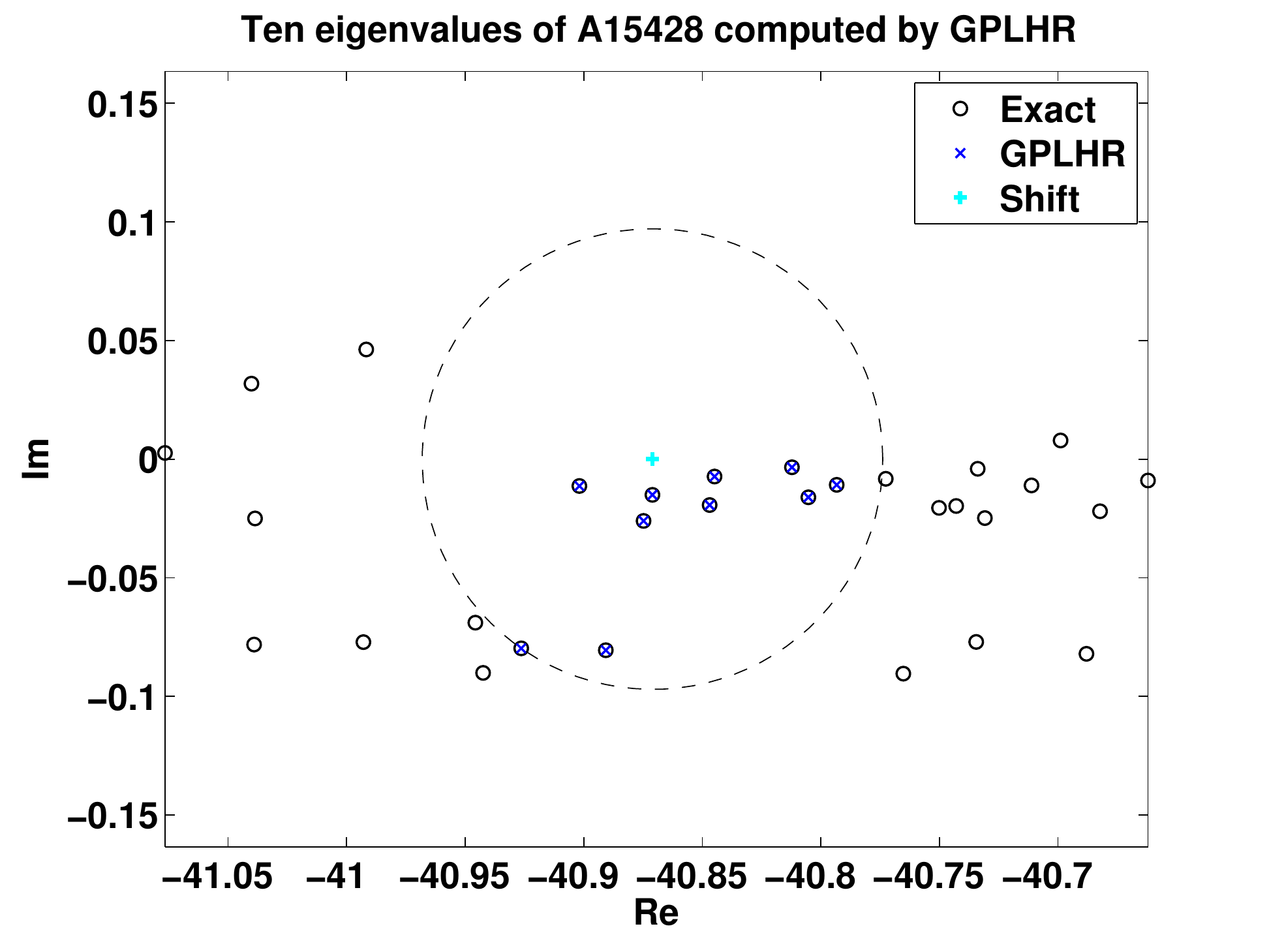}
    \includegraphics[width=6.4cm]{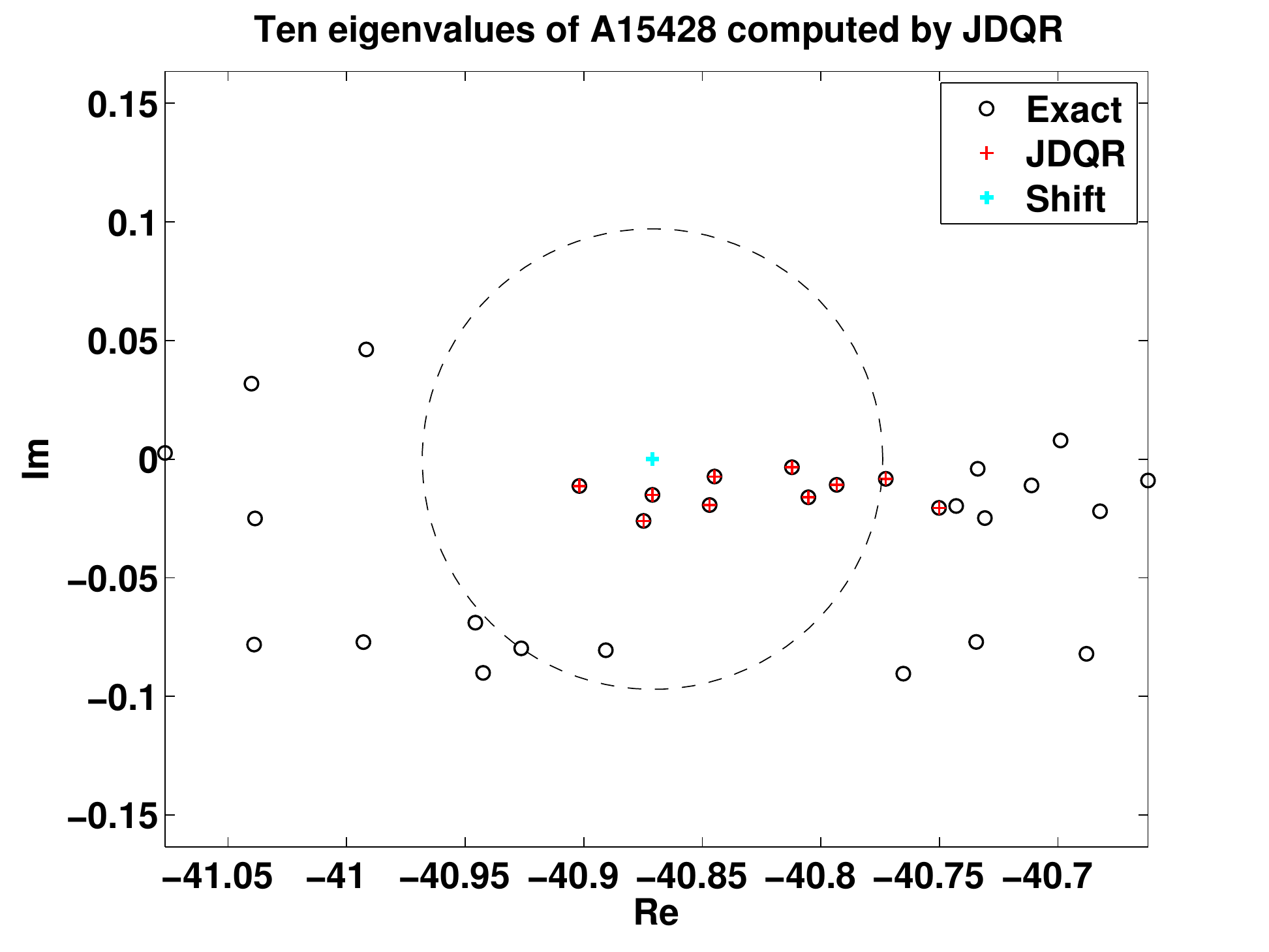}\\
    \includegraphics[width=6.4cm]{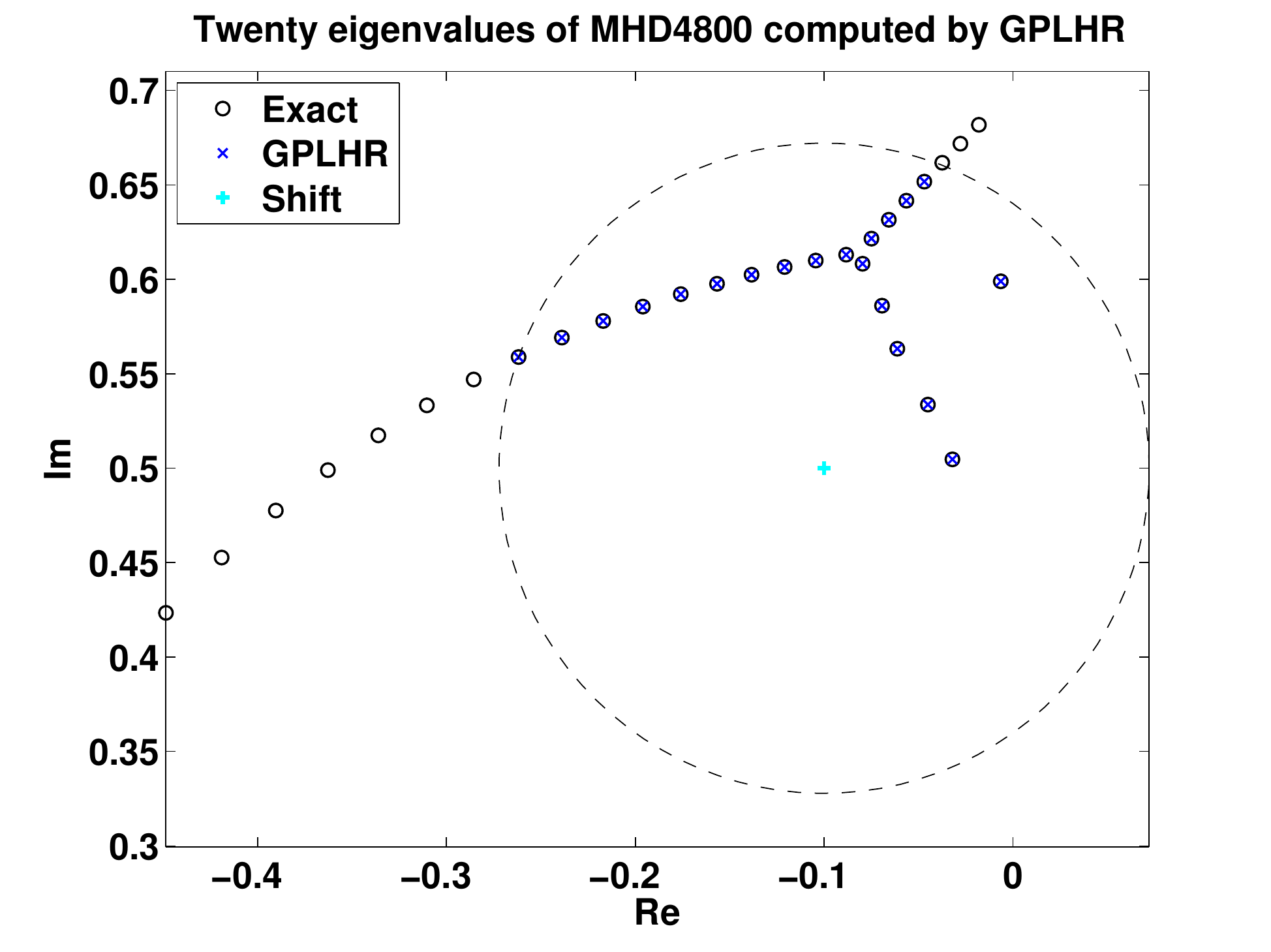}
    \includegraphics[width=6.4cm]{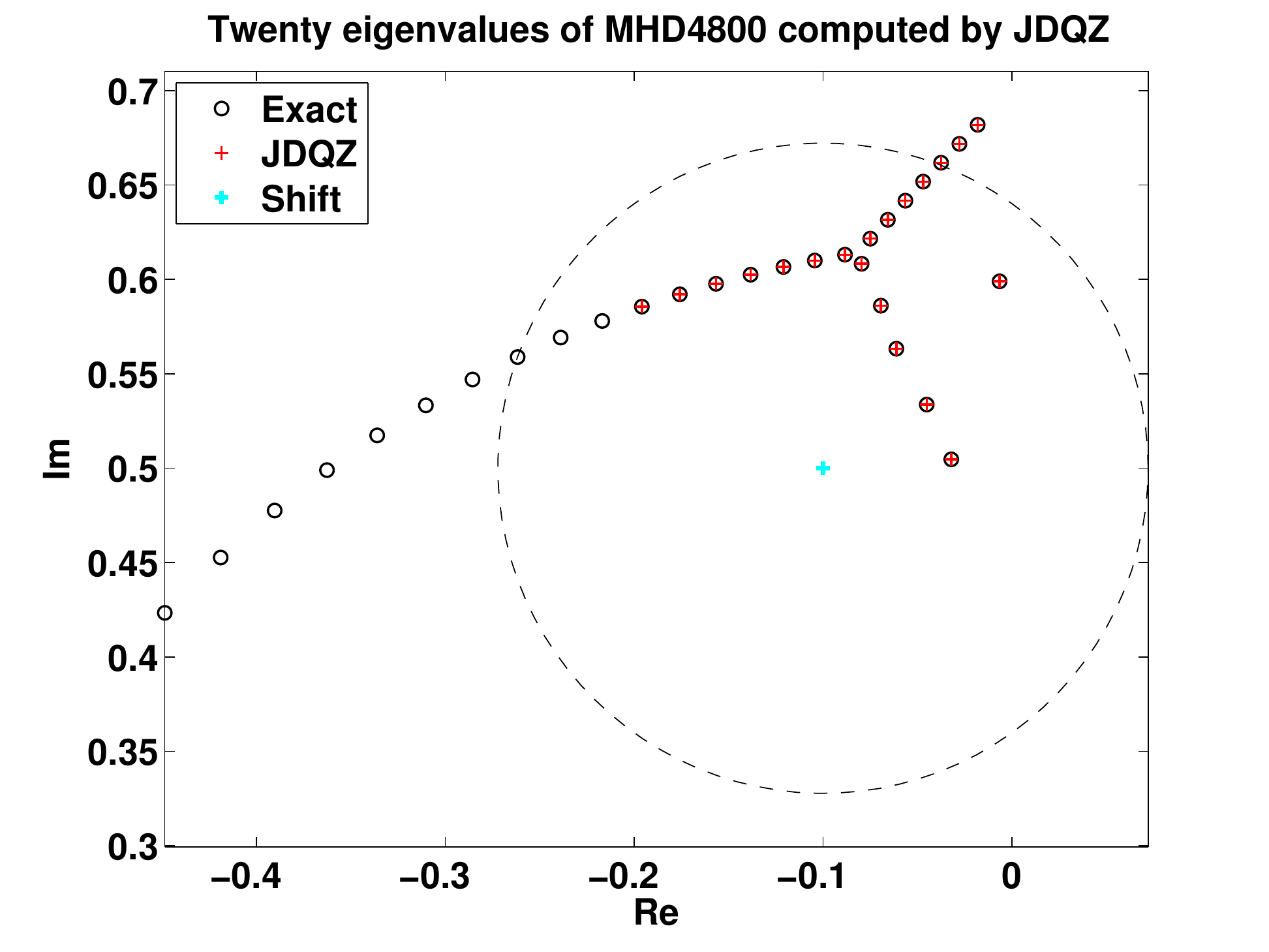}
\end{center}
\caption{
{\small 
Ten eigenvalues of A15428 closest to $\sigma = -40.8710$ (top) and twenty eigenvalues of MHD4800 (Alfv\'{e}n branch)
closest to $\sigma = -0.1+0.5i$ (bottom) computed by GPLHR (left) and JDQR/JDQZ (right) algorithms.
}
}
\label{fig:misconv}
\end{figure}


We would like to emphasize the remarkable robustness of GPLHR. Given an appropriate preconditioner $T$, in all of our tests,
the method consistently obtained the correct results, without convergence stagnation or misconvergence,
as was observed in several JDQR/JDQZ runs (in Tables~\ref{tab:jdqr} and~\ref{tab:jdqz}, ``DNC'' means that the method did not converge
and ``MCV'' denotes misconvergence).
The observed robustness is not surprising, as the block methods have traditionally been recommended for properly resolving eigenvalue clusters~\cite{Parlett:98}.   

We have also noticed that JDQZ/JDQR are more likely to misconverge, i.e., discover incorrect eigenvalues, 
when a larger number of eigenpairs (say, $k \geq 10$) is needed. Such situations are demonstrated in Figure~\ref{fig:misconv}, 
where $k = 10$ eigenpairs of A15428 (top) and $k = 20$ eigenpairs of MHD4800 (bottom) closest to $\sigma$ are computed by different
eigensolvers. 
It can be seen that the GPLHR found all the targeted eigenvalues (within the
dashed circle), whereas JDQR/JDQZ missed a few of them. This can be possibly caused by the repeated deflations
performed by JDQR/JDQZ, which can accumulate error and lead to an inaccurate solution. As a GPLHR preconditioner $T$ for A15428, we use an approximate
solve of $(A - \sigma I) w = r$ by GMRES preconditioned with ILU($10^{-2}$) and convergence tolerance (\textit{tol\_prec}) of~$10^{-3}$. 
The GMRES solve with the same ILU($10^{-2}$) preconditioner is applied to the correction equation in JDQR. 
For the MHD4800 case, the GPLHR preconditioner is $T = (A - \sigma B)^{-1}$, which is also applied to the residuals in the JDQZ scheme to define corrections,
without performing GMRES steps. Note that adding inner GMRES iterations to solve the correction equation did not fix the JDQZ misconvergence issue.



\subsection{Comparison with ARPACK}

In our last set of tests we compare GPLHR with the implicitly restarted Arnoldi method available in ARPACK through
{\sc matlab}'s \texttt{eigs} function. Here, we restrict attention to the top three examples in Table~\ref{tab:test}
which are among the largest in our collection. The results for these problems are representative for large-scale eigenproblems.

In order to compute $k$ eigenvalues 
closest to the target shift $\sigma$,
ARPACK has to be supplied with a procedure to invert the shifted matrix $A - \sigma B$, which represents a major challenge in this type
of calculation. In essence, ARPACK applies the Arnoldi algorithm to the transformed matrix pair $((A - \sigma B)^{-1} B, I)$, and therefore 
the inversion of $A - \sigma B$ should be performed to high accuracy to maintain exact Arnoldi relations.

By contrast, GPLHR does not require any exact inverses and instead relies on preconditioning, given by 
$T \approx (A - \sigma B)^{-1}$, which, as we have observed, does not necessarily have to provide a very accurate approximation
of the shift-and-invert operator. Our tests demonstrate that the savings obtained by GPLHR through avoiding the exact inverse of $A - \sigma B$ can be substantial.

One way to invert $A - \sigma B$ is by performing its LU factorization. 
Then the matrix-vector products in ARPACK can be 
defined through triangular solves with the $L$ and $U$ factors. However, 
depending on the sparsity and structure, 
the full LU factorization can be costly, and is generally significantly 
more expensive than performing it incompletely in the context of
ILU preconditioning.

\begin{table}[htb]
\centering
{\small
\begin{tabular}{|c|c|c|c|}
\hline 
Problem & $k$ & GPLHR & ARPACK\tabularnewline
\hline 
\hline 
A15428 & 10 & 335 & 3006\tabularnewline
\hline 
AF23560 & 5 & 2 & 8\tabularnewline
\hline 
CRY10000 & 5 & 1.4 & 1.9\tabularnewline
\hline 
\end{tabular}
}
\caption{
{\small Time required by GPLHR with the preconditioner $T$ given by ILU\textup{($10^{-3}$)} and ARPACK with the shift-and-invert
operator computed through the LU decomposition of $A - \sigma I$
to find $k$ eigenvalues closest to $\sigma$ and their associated eigenvectors.}}
\label{tab:arpack}
\end{table}

Table~\ref{tab:arpack} shows the difference in time required by GPLHR and ARPACK to compute a number of eigenvalues closest to~$\sigma$.
Here, the shift-and-invert in ARPACK is performed by means of the LU factorization, whereas GPLHR only uses the incomplete ILU factors as
a preconditioner. One can see that, even for problems of the size of several tens of thousands, GPLHR can achieve almost 10 time speed up. 
Note that the results reported in Table~\ref{tab:arpack} include time required to compute the full and incomplete triangular factors in ARPACK and 
GPLHR, respectively.   
All of our results were obtained on an iMac computer running Mac OS X 10.8.5, MATLAB R2013a, 
with a 3.4 GHz Intel Core i7 processor and 16GB 1600MHz DDR3 memory.

\begin{table}[htb]
\centering
{\small
\begin{tabular}{|c|c|c|c|c|}
\hline 
Problem & $k$ & \textit{$it_{G}$} & GPLHR & ARPACK\tabularnewline
\hline 
\hline 
A15428 & 10 & 7 & 153 & 196\tabularnewline
\hline 
AF23560 & 5 & 5 & 6 & 131\tabularnewline
\hline 
CRY10000 & 5 & 5 & 6 & 95\tabularnewline
\hline 
\end{tabular}
}
\caption{
\small{Time required by GPLHR with the preconditioner $T$ given by \textit{$it_{G}$} steps of preconditioned GMRES and ARPACK with shift-and-invert 
operator defined through a preconditioned GMRES solve of $(A - \sigma I)w = r$ to full accuracy to compute $k$ eigenvalues closest to $\sigma$ 
and their associated eigenvectors. In all cases, 
GMRES is preconditioned with~ILU\textup{($10^{-2}$)}.   
%
}
}
\label{tab:arpack_iter}
\end{table}

Another way to perform shift-and-invert in ARPACK is through invoking an iterative linear solver to compute the solution of a system 
$(A - \sigma B)w = r$ to full accuracy. Instead, GPLHR can rely on a solve with a significantly more relaxed
stopping criterion, which will play a role of the preconditioner $T$. The performance of ARPACK and GPLHR in such a setting is
compared in Table~\ref{tab:arpack_iter}, which reports time required by both solvers to obtain the solution. 
Here, the ARPACK's shift-and-invert is performed using the full-accuracy GMRES solve, preconditioned by ILU($10^{-2}$). 
At the same time, the GPLHR preconditioner $T$ is 
given by only a few ($\textit{$it_{G}$}$) GMRES steps with the same ILU($10^{-2}$) preconditioner. 
Once again, in these tests, we observe more than 10 time speedup when the eigensolution is performed using GPLHR. 




\section{Conclusions}\label{sec:con}
We have presented the Generalized Preconditioned Locally Harmonic Residual method (GPLHR) 
for computing eigenpairs of non-Hermitian regular matrix pairs $(A,B)$ that 
correspond to the eigenvalues closest to a given shift~$\sigma$. 
GPLHR is a block eigensolver, 
which takes advantage
of a preconditioner when it is available, and does not require an exact or highly accurate shift-and-invert procedure.
The method allows multiple matrix-vector product to be computed simultaneously,
and can take advantage of BLAS3 through blocking. As a result, it is 
suitable for high performance computers with many processing units.
 
Our numerical experiments demonstrated the robustness and 
reliability of the proposed algorithm. We compared GPLHR 
to several state-of-the-art eigensolvers (including block GD, JDQR/JDQZ, and implicitly restarted Arnoldi) 
for a variety of eigenvalue problems coming from different applications.
Our results show that GPLHR is competitive to the established approaches 
in general.  It is often more efficient, especially if there is a limited
amount of memory.

\newpage
\appendix
\section{The GPLHR-EIG algorithm}\label{sec:gplhr-eig}
In Algorithm~\ref{alg:gplhr-eig}, we present 
an eigenvector-based version of the GPLHR algorithm, called 
GPLHR-EIG.
\begin{algorithm}[!htbp]
\begin{small}
\begin{center}
  \begin{minipage}{5in}
\begin{tabular}{p{0.5in}p{4.5in}}
{\bf Input}:  &  \begin{minipage}[t]{4.0in}
                 A pair $(A, B)$ of $n$-by-$n$ matrices, shift $\sigma \in \IC$ different from any eigenvalue of $(A,B)$, 
                 preconditioner $T$, starting guess of eigenvectors $X^{(0)} \in \IC^{n \times k}$, and the subspace expansion parameter $m$. 
                  \end{minipage} \\
{\bf Output}:  &  \begin{minipage}[t]{4.0in}
                  If $B \neq I$, then approximate eigenvectors $X \in \IC^{n \times k}$ and the associated diagonal matrices 
                  $\Lambda_A, \Lambda_B \in \IC^{k \times k}$ in~\eqref{eq:partial}, such that $\lambda_j = \Lambda_A(j,j)/\Lambda_B(j,j)$
                  are the $k$ eigenvalues of $(A,B)$ closest to $\sigma$. \\  
                  If $B = I$, then approximate eigenvectors $X \in \IC^{n \times k}$ and the associated diagonal matrix 
                  $\Lambda \in \IC^{k \times k}$ of $k$ eigenvalues of $A$ closest to $\sigma$. \\  
                  \end{minipage}
\end{tabular}
\begin{algorithmic}[1]
\STATE $X \gets X^{(0)}$; $Q \gets (A - \sigma B)X$; $P \gets [ \ ]$;
\STATE Normalize columns of $X$ ($x_j \gets x_j/\|x_j\|$); $\Lambda_A \gets \mbox{diag}\left\{ X^* A X \right\}$;  $\Lambda_B \gets \mbox{diag}\left\{ X^* B X \right\}$;
\WHILE {convergence not reached}
  \STATE $V \gets \text{orth}(X)$; $Q \gets \text{orth}(Q)$; 
  \IF{$B \neq I$}
     \STATE $W \gets (I - VV^*)T(I -QQ^*)(AX\Lambda_B - B X\Lambda_A)$;
  \ELSE
     \STATE $W \gets (I - VV^*)T(I -VV^*)(AX\Lambda_B - X\Lambda_A)$;
  \ENDIF 
  \STATE $W \gets \mbox{orth}(W)$; $S_{0} \gets W$; $S \gets [ \ ]$;
  \FOR {$l = 1 \rightarrow m$}
     \IF{$B \neq I$}
        \STATE $S_{l} \gets (I - VV^*)T(I - QQ^*)(A S_{l-1} \Lambda_B - B S_{l-1} \Lambda_A)$;
     \ELSE
        \STATE $S_{l} \gets (I - VV^*)T(I - VV^*)(A S_{l-1} \Lambda_B - S_{l-1} \Lambda_A)$;
     \ENDIF
     \STATE $S_l \gets S_l -  W( W^*S_l)$; $S_l \gets S_l - S(S^*S_l)$; 
$S_l \gets \mbox{orth}(S_l)$; $S \gets [S \ S_l]$;
  \ENDFOR
  \STATE $P \gets P - V(V^*P)$; $P \gets P - W(W^*P)$; $P \gets P - S(S^*P)$; $P \gets \text{orth}(P)$;
  \STATE Set the trial subspace $Z \gets [V, \ W, \  S_{1}, \ldots,  \ S_{m}, \ P ]$;
  \STATE $\hat Q \gets \texttt{orth}( (A - \sigma B) [W, \ S_1, \ \ldots, \ S_m, P] )$; $\hat Q \gets \hat Q - Q(Q^*\hat Q)$; 
  \STATE Set the test subspace $U \gets [Q, \hat Q]$;
  \STATE Solve the projected eigenproblem $(U^* A Z, U^* B Z)$; 
   set $\bar Y$ to be the matrix of all eigenvectors ordered  according to their eigenvalues' closeness to $\sigma$, ascendingly;
  \STATE $Y \gets \bar Y(\texttt{:,1:k})$; 
  \STATE  $P \gets W Y_W  + S_{1} Y_{S_1} + \ldots + S_{m} Y_{S_m} + P Y_{P}$, where $Y \equiv [Y_V^T, Y_W^T , Y_{S_1}^T, \ldots, Y_{S_m}^T, Y_P^T ]^T$
  is a conforming partition of $Y$;
  \STATE  $X \gets V Y_V + P$; $Q \gets (A - \sigma B)X$;
  \STATE $P \gets Z \bar Y(\texttt{:,k+1:2k}))$;\footnote{This step can be disabled if the LOBPCG-style search direction of step 25 is preferred.}
  \STATE  Normalize columns of $X$ ; 
$\Lambda_A \gets \mbox{diag}\left\{ X^* A X \right\}$;  $\Lambda_B \gets \mbox{diag}\left\{ X^* B X \right\}$;
\ENDWHILE
  \STATE If $B = I$, then $\Lambda \gets \Lambda_A$.
\end{algorithmic}
\end{minipage}
\end{center}
\end{small}
  \caption{The GPLHR-EIG algorithm for partial eigendecomposition~\eqref{eq:partial}}
  \label{alg:gplhr-eig}
\end{algorithm}

The work flow of Algorithm~\ref{alg:gplhr-eig} is nearly identical to  that of the Schur vector based 
GPLHR in Algorithm~\ref{alg:gplhr-qz}. The main algorithmic difference is that GPLHR-EIG discards the 
harmonic Ritz values computed in step 23 and, instead, defines eigenvalue approximations 
$(\alpha_j, \beta_j)$ using bilinear forms
$\alpha_j = x_j^{*}A x_j$ and $\beta_j = x_j^{*}B x_j$, 
where $x_j$ are the columns of $X$ at a given iteration ($j = 1, \ldots, k$). Clearly, with this choice,
the ratios $\alpha_j/\beta_j$ are exactly the Rayleigh quotients for $(A, B)$ evaluated at the 
corresponding harmonic Ritz vectors. This approach is motivated by the fact that the Rayleigh quotients
represent optimal eigenvalue approximations, and is common in the harmonic Rayleigh--Ritz based
eigenvalue computations;~e.g.,~\cite{Morgan:91, Morgan.Zeng:98, Ve.Kn:14TR}. 
 

The complexity and memory requirements of Algorithm~\ref{alg:gplhr-eig} are comparable to those of Algorithm~\ref{alg:gplhr-qz}. 
Note that it is not necessary to keep both the approximate eigenvectors $X$ 
and the orthonormal basis $V$, since
$X$ can be rewritten with $V$.
Therefore, no additional memory for storing $V$ is needed in practice.   

Although in the numerical experiments of Section~\ref{sec:num} 
we report only the results for the Schur vector based variant of GPLHR (Algorithm~\ref{alg:gplhr-qz}),
the performance of GPLHR-EIG was found to be similar for the problems under considerations. 
Therefore, we do not report results for Algorithm~\ref{alg:gplhr-eig}.
Instead we refer the reader to~\cite{Zuev.Ve.Yang.Orms.Krylov:15}, where a version of GPLHR-EIG has been benchmarked for a specific application.

Finally, let us remark that GPLHR-EIG 
is not quite suitable for computing larger subsets 
of eigenpairs using the deflation technique. In contrast to the Schur vector based GPLHR, the computed set of eigenvectors cannot 
be directly used for deflation. Thus, Algorithm~\ref{alg:gplhr-eig} should be invoked in situations where the number of desired eigenpairs $k$ is sufficiently
small to ensure their efficient computation in a single run of the algorithm with the block size (at least) $k$.

\bibliographystyle{siam} 
\bibliography{eig,plmr}

\end{document}